\documentclass[12pt,leqno,a4paper]{article}
\usepackage{amssymb}
\usepackage{amsmath}
\usepackage{amsthm}
\usepackage{color}
\usepackage[pdfborder={0 0 0}]{hyperref}
\usepackage[shortlabels]{enumitem}
\usepackage{mathrsfs}							
\usepackage{bbm}									
\usepackage{esint}								

\newtheorem{theorem}{Theorem}
\newtheorem{proposition}[theorem]{Proposition}
\newtheorem{lemma}[theorem]{Lemma}
\newtheorem{corollary}[theorem]{Corollary}

\theoremstyle{remark}

\theoremstyle{definition}
\newtheorem{remark}[theorem]{Remark}
\newtheorem{definition}[theorem]{Definition}

\numberwithin{equation}{section}
\numberwithin{theorem}{section}

\newcommand\set[1]{\left\{\,#1\,\right\}}
\newcommand\abs[1]{\left|#1\right|}

\newcommand\norm[1]{\left\Vert#1\right\Vert}
\newcommand{\lamax}{\lambda_{\text{max}}}

\newcommand{\R}{\mathbb{R}}

\newcommand{\N}{\mathbb{N}}

\newcommand{\T}{\mathbb{T}}

\newcommand{\cA}{{\mathcal A}}

\newcommand{\cC}{{\mathcal C}}

\newcommand{\cF}{{\mathcal F}}

\newcommand{\cH}{{\mathcal H}}
\newcommand{\cI}{{\mathcal I}}
\newcommand{\cJ}{{\mathcal J}}

\newcommand{\cS}{{\mathcal S}}

\DeclareMathOperator{\id}{id}

\DeclareMathOperator{\tr}{tr}

\DeclareMathOperator{\curl}{curl}
\DeclareMathOperator{\divv}{div}

\begin{document}
\title{The Rayleigh-Taylor instability with local energy dissipation}
\author{Bj\"orn Gebhard \and J\'ozsef J. Kolumb\'an}
\date{}
\maketitle

\begin{abstract}
We consider the inhomogeneous incompressible Euler equations including their local energy inequality as a differential inclusion. Providing a corresponding convex integration theorem and constructing subsolutions, we show the existence of locally dissipative Euler flows emanating from the horizontally flat Rayleigh-Taylor configuration and having a mixing zone which grows quadratically in time. For the Rayleigh-Taylor instability these are the first turbulently mixing solutions known to respect local energy dissipation, and outside the range of Atwood numbers considered in \cite{GKSz}, the first weakly admissible solutions in general. In the coarse grained picture the existence relies on one-dimensional subsolutions described by a family of hyperbolic conservation laws, among which one can find the optimal background profile appearing in the scale invariant bounds from \cite{Kalinin_Menon_Wu}, and as we show, the optimal conservation law with respect to maximization of the total energy dissipation.  
\end{abstract}

\section{Introduction}

The Rayleigh-Taylor instability is one of the most classical hydrodynamic instabilities \cite{Rayleigh,Taylor}, occurring whenever two fluids of different densities are accelerated into each other in the ``wrong'' direction. Such an acceleration can be due to a gravitational force acting on two layered incompressible fluids with the heavier one on top of the lighter one, i.e. for instance in the flat resting two phase configuration with density and velocity given by
\begin{align}\label{eq:flat_initial_data}
    \rho_0(x)=\begin{cases}
    \rho_+,&x_3>0,\\
    \rho_-,&x_3<0,
    \end{cases}\quad v_0(x)=0,
\end{align}
where $0<\rho_-<\rho_+$ are the homogeneous densities of the two fluids. In real world experiments and numerical simulations (we refer to the surveys \cite{Boffetta_Mazzino_survey,sharp_survey,zhou_survey_1,zhou_survey_2}) it is observed that the initially separated fluids begin to turbulently mix in a zone with upper and lower bounds given by the sets
\begin{align*}
    \set{x:x_3=\alpha_\pm gA t^2}.
\end{align*}
Here $g>0$ denotes the gravity constant, $A:=\frac{\rho_+-\rho_-}{\rho_++\rho_-}$ is the Atwood number, and $\alpha_\pm>0$, $\alpha_+$ for the upper boundary, $\alpha_-$ for the lower, are constants to be sought through experiments and simulations.

In the present article we consider ideal fluids described by the incompressible Euler equations with local energy (in)equality, these are
\begin{align}\label{eq:euler_only}
\begin{split}
\partial_t(\rho v)+\text{div }(\rho v\otimes v) + \nabla p &= - \rho g e_n,\\
\text{div } v&=0,\\
\partial_t \rho + \text{div }(\rho v)&=0,
\end{split}
\end{align}
as well as
\begin{align}\label{eq:euler_local_energy_inequality}
    \partial_t \left(\frac{1}{2}\rho|v|^2+\rho g x_n\right)+\divv \left(\left(\frac{1}{2}\rho|v|^2+\rho g x_n+p\right)v\right)&\leq 0.
\end{align}
The equations are considered on $\Omega\times(0,T)$, where $\Omega$ is a bounded domain in $\R^n$, $n=2,3$, or a periodic channel $\T^{n-1}\times(-L,L)$, and $T>0$.
The unknown functions are the mass density $\rho:\Omega\times(0,T)\rightarrow \R$, the velocity field $v:\Omega\times(0,T)\rightarrow \R^n$ and the pressure $p:\Omega\times(0,T)\rightarrow \R$, whereas $g>0$ is the given gravity constant and $e_n$ is the $n$-th unit vector. The left-hand side of \eqref{eq:euler_local_energy_inequality} defines the local energy dissipation measure $\nu$.
The divergence-free condition is augmented by the usual no-penetration boundary condition 
\begin{align}\label{eq:boundary_condition_v}
    v\cdot \vec{n} =0\text{ on }\partial\Omega\times(0,T).
\end{align}
In the case of a periodic channel the latter condition is meant to be imposed on $\T^{n-1}\times\{\pm L\}$, where $\vec{n}=\pm e_n$. 

We emphasize that equations \eqref{eq:euler_only} automatically imply inequality \eqref{eq:euler_local_energy_inequality} to hold with equality provided the considered solution is regular enough, see for example Feireisl et al. \cite{feireisl_etal_energy_conservation}, and also \cite{chen_yu_onsager,constantin_etal_onsager,leslie_shvydkoy_onsager} for other related Onsager-type statements.

Solutions to which this applies emerge from the local in time well-posedness results for initial data in sufficiently high H\"older, Sobolev or Besov spaces \cite{bei_valli_I_local_wellposedness,bei_valli_II_local_wellposedness,danchin_local_wellposedness,danchin_fanelli_local_wellposedness,marsden_local_wellposedness,valli_zajac_local_wellposedness}. The spaces considered in \cite{bei_valli_I_local_wellposedness,bei_valli_II_local_wellposedness,danchin_local_wellposedness,danchin_fanelli_local_wellposedness,marsden_local_wellposedness,valli_zajac_local_wellposedness} all have the property that they embed into the space of Lipschitz continuous functions. In contrast, the initial data of our interest clearly is not continuous. 

While the Rayleigh-Taylor instability has also been studied in well-posed settings, for instance by a smooth approximation of $\rho_0$ or by adding surface tension or viscosity to the model \cite{hwang_guo_RT,jiang_jiang_zhan_RT,lian_RT,pruess_simonett_RT}, the initial value problem as stated in \eqref{eq:flat_initial_data}, \eqref{eq:euler_only}, \eqref{eq:boundary_condition_v} is ill-posed and naturally leads one to use weak solutions. Using convex integration, Sz\'ekelyhidi and the authors managed to give in \cite{GKSz} the first existence result of weak solutions to \eqref{eq:euler_local_energy_inequality}, \eqref{eq:boundary_condition_v} emanating from \eqref{eq:flat_initial_data} and having a mixing zone which grows quadratically in time.

Contrary to classical solutions, it is known that weak solutions of \eqref{eq:euler_only} in general can be unphysical in the sense that they may violate the energy dissipation inequality \eqref{eq:euler_local_energy_inequality}. This can happen in a way such that even a weak total energy admissibility such as
\begin{align}\label{eq:weak_energy_admissibility}
    E(t)\leq E(0)\text{ for almost every }t\in(0,T),
\end{align}
where $E$ denotes the total energy
\begin{align*}
    E(t):=\int_\Omega \frac{1}{2}\rho(t,x)\abs{v(t,x)}^2+\rho(t,x)gx_n\:dx,
\end{align*}
is violated. Note that \eqref{eq:weak_energy_admissibility} indeed follows from \eqref{eq:euler_local_energy_inequality}, \eqref{eq:boundary_condition_v}. Consequently, weak solutions have to be constructed in a way such that their local, or at least their global, energy dissipation can be controlled.  
In \cite{GKSz} it was shown that the constructed solutions satisfy the weak admissibility condition \eqref{eq:weak_energy_admissibility} provided the Atwood number is in a sufficiently high range. One of the main goals of the present article is to strengthen the machinery for the construction of weak solutions to \eqref{eq:euler_only}, i.e. the corresponding convex integration theorem, in a way such that the obtained solutions automatically satisfy not only the global admissibility \eqref{eq:weak_energy_admissibility}, but also the local energy inequality \eqref{eq:euler_local_energy_inequality}. 

As in \cite{GK-EE} by the authors for the homogeneous Euler equations and in \cite{Markfelder_EE} by Markfelder for compressible Euler, this improvement is achieved by directly including \eqref{eq:euler_local_energy_inequality} and corresponding variables into the differential inclusion describing the Euler equations. During our investigation we establish the full relaxation of the differential inclusion in the unbounded case and a large part of it in the bounded case. This large part is sufficient to state a typical $h$-principle or convex integration theorem for system \eqref{eq:euler_only}, \eqref{eq:euler_local_energy_inequality}. We refer to Section \ref{sec:statement_of_results} for the detailed statements.

Concerning actual examples of Euler solutions that obey the local energy dissipation inequality, we utilize the improved convex integration theorem to give the following statement for the Rayleigh-Taylor configuration \eqref{eq:flat_initial_data}. For the precise notion of solution we refer to Definition \ref{def:weaksols}.

Let $F_\lambda:[\rho_-,\rho_+]\rightarrow\R$, $\lambda>0$, $\rho_+>\rho_->0$, be the uniformly convex function
\begin{align}\label{eq:F_in_thm_in_introduction}
    F_\lambda(r):=-\lambda \frac{(\rho_+-r)(r-\rho_-)}{\rho_++\rho_--r},
\end{align}
and let $\bar{\rho}_\lambda(x_n,t)$ be the unique entropy solution of the Riemann problem
\begin{align}\label{eq:hyperbolic_law_in_theorem_in_introduction}
    \partial_t\bar{\rho}+gt\partial_{x_n} (F_\lambda(\bar{\rho}))=0,\quad \bar{\rho}(\cdot,0)=\rho_0.
\end{align}
\begin{theorem}\label{thm:introduction}
Let $\rho_+>\rho_->0$, $\Omega=\T^{n-1}\times(-L,L)$, $L>0$, $\abs{\T}=1$, $n=2,3$ and $\lambda\in (0,1/2)$. There exist infinitely many weak solutions $(\rho,v)\in L^\infty(\Omega\times(0,T);\R\times \R^n)$ of \eqref{eq:flat_initial_data}, \eqref{eq:euler_only}, \eqref{eq:euler_local_energy_inequality},  \eqref{eq:boundary_condition_v} with local dissipation measure arbitrarily close to
\begin{align}\label{eq:local_dissipation_measure_in_thm_in_introduction}
    \nu_\lambda=(1-2\lambda)g^2tF_\lambda(\bar{\rho}_\lambda).
\end{align}
The solutions satisfy $\rho\in\set{\rho_-,\rho_+}$ almost everywhere and are induced by  $\bar{\rho}_\lambda$ as their common horizontally averaged profile (subsolution). Moreover, their total energy balance is (again up to an arbitrarily small error) given by
\begin{align}\label{eq:total_balance_in_thm_in_introduction}
\frac{E(t)-E(0)}{\rho_+-\rho_-}=\lambda^2\left(\lambda-\frac{1}{2}\right)\frac{g^3t^4A^2}{3(1-A^2)}.
\end{align}
\end{theorem}

\begin{remark}
a) Normalized to $\bar{s}_\lambda\in[-1,1]$ via $\bar{\rho}_\lambda(x_n,t)=:\frac{\rho_++\rho_-}{2}(1+\bar{s}_\lambda(x_n,t)A)$, and stated in explicit form there holds
\begin{align}\label{eq:rho_bar_explicit}
    \bar{s}_\lambda\left(\frac{\lambda  Agt^2}{2}\xi,t\right)=\frac{1}{A}\left(1-\sqrt{\frac{1-A^2}{1+A\xi}}\right),\quad \xi\in \left(\frac{-2}{1+A},\frac{2}{1-A}\right),
\end{align}
as well as $\bar{s}_\lambda\left(\frac{\lambda  Agt^2}{2}\xi,t\right)=1$ for $\xi\geq \frac{2}{1-A}$ and $\bar{s}_\lambda\left(\frac{\lambda A gt^2}{2}\xi,t\right)=-1$ for $\xi\leq \frac{-2}{1+A}$.
In particular, the mixing zone at time $t$ associated with the solutions is given by
\begin{align*}
    \set{x\in\Omega:\frac{-\lambda}{1+A}Agt^2<x_n<\frac{\lambda}{1-A}Agt^2}.
\end{align*}

b) The time $T=T(\lambda,g,A,L)>0$ is chosen such that the mixing zone does not reach the top and bottom boundaries $\T^{n-1}\times\{\pm L\}$.

c) Clearly $\nu_\lambda\leq 0$. For clarification, the actual $\nu$, arbitrarily close to $\nu_\lambda$, also satisfies $\nu\leq 0$. 
\end{remark}

We would like to point out that the above family of conservation laws, and thus the family of their induced solutions, has two prominent members. First of all, one sees that the limiting case $\lambda\rightarrow \frac{1}{2}$ gives rise to a perfect global \eqref{eq:total_balance_in_thm_in_introduction} and local \eqref{eq:local_dissipation_measure_in_thm_in_introduction} energy balance. In fact the balance law \eqref{eq:hyperbolic_law_in_theorem_in_introduction} with $\lambda=1/2$ also arises as the optimal background profile in the scale invariant bounds obtained by Kalinin, Menon, Wu \cite{Kalinin_Menon_Wu}. These bounds rely on energy conservation. Computing, cf. \eqref{eq:computation_E_pot}, the increase in potential energy of our solutions for $\lambda=1/2$ we indeed find
\begin{align}\label{eq:potential_energy_KMW}
    \frac{E_{pot}(0)-E_{pot}(t)}{\rho_+-\rho_-}=\frac{g^3t^4A^2}{24(1-A^2)},
\end{align}
and therefore deduce the following statement.
\begin{corollary}\label{cor:optimal_bounds}
In the immiscible limit the scale invariant bound for the decay of potential energy \cite[Theorem 2.1]{Kalinin_Menon_Wu} is sharp.\footnote{For comparison  we would like to remark that \cite[Equation (2.8)]{Kalinin_Menon_Wu} indeed is in agreement with \eqref{eq:potential_energy_KMW}, but in the definition of the specific background profile \cite[Equation (3.3)]{Kalinin_Menon_Wu}  that realizes the stated bound an additional factor $1/2$ is missing when defining $\tau$.}
\end{corollary}

The connection to \cite{Kalinin_Menon_Wu} is not the only reason for considering the family $F_\lambda$. It is clear from \eqref{eq:total_balance_in_thm_in_introduction} that $\lambda =\frac{1}{3}$ is maximizing the total energy dissipation within the family $F_\lambda$. However, it turns out that this dissipation is neither beaten by any other convex $F$ outside the family.
\begin{proposition}\label{prop:lambda_1_3}
Among all one-dimensional (potentially non-strict) subsolutions with $\bar{\rho}$ given as the entropy solution to a Riemann problem 
\begin{align*}
    \partial_t\bar{\rho}+gt\partial_{x_n} (F(\bar{\rho}))=0,\quad \bar{\rho}(\cdot,0)=\rho_0,
\end{align*}
with $F\in\cC^2([\rho_-,\rho_+])$ uniformly convex and $F(\rho_\pm)=0$, the induced total energy dissipation $E(0)-E(t)$ is maximal for $F=F_{\frac{1}{3}}$.
\end{proposition}

For the precise notion of one-dimensional subsolutions, which can be seen as the horizontally averaged profiles $\bar{\rho}$, and in general more precise statements, we refer to Sections \ref{sec:statement_of_results}, \ref{sec:statement_of_results_2}.

\subsection{Change of coordinates}\label{sec:change_of_coordinates}

In order to formulate our results it is more convenient to transform the equations to an accelerated domain in which the contribution of the gravitational body force is eliminated. For that we define 
\begin{align}\label{eq:definition_accelerated_domain}
\mathscr{D}:=\set{(y,t)\in\R^n\times(0,T):y-\frac{1}{2}gt^2e_n\in\Omega},
\end{align}
as well as for $t\in(0,T)$ the slice  
\[
\mathscr{D}(t):=\set{y\in\R^n:(y,t)\in\mathscr{D}}=\Omega+\frac{1}{2}gt^2e_n.
\]
Then, the inhomogeneous Euler equations \eqref{eq:euler_only} stated on $\Omega\times(0,T)$ for the functions $(\rho,v,p)$ are equivalent to 
\begin{align}\label{eq:main0g}
\begin{split}
\partial_t(\mu w)+\text{div }(\mu w\otimes w) + \nabla q &= 0,\\
\text{div } w&=0,\\
\partial_t \mu + \text{div }(\mu w)&=0,
\end{split}
\end{align}
stated on the spacetime set $\mathscr D$, where $(\mu,w,q)$ are defined 
 by means of the coordinate transformation
$
    y:=x+\frac{1}{2}gt^2e_n,~x\in\Omega,~t\in(0,T)
$
and
\begin{align}\label{eq:transf}
\begin{split}
\rho(x,t)&=\mu\left(y,t\right),\\
v(x,t)&=w\left(y,t\right)-g t e_n,\\
p(x,t)&=q\left(y,t\right).
\end{split}
\end{align}
In addition, the boundary condition \eqref{eq:boundary_condition_v} translates to 
\begin{align}\label{eq:boundary_condition_w}
    (w(\cdot,t)-gte_n)\cdot \vec{n}=0~\text{on }\partial\mathscr{D}(t),~t\in(0,T),
\end{align}
whereas the initial data remains untouched, i.e. $\mu(\cdot,0)=\rho(\cdot,0)$, $v(\cdot,0)=w(\cdot,0)$.
It can be checked that this equivalence holds true for weak solutions, cf. Appendix \ref{sec:defs}. For the sake of readability we have also moved all precise definitions concerning weak solutions to Appendix \ref{sec:defs}.

Concerning the dissipation inequality, we also have that \eqref{eq:euler_local_energy_inequality} on $\Omega\times (0,T)$ is equivalent to
\begin{align}\label{eq:euler_0g_local_energy_inequality}
    \partial_t \left(\frac{1}{2}\mu|w|^2\right)+\divv \left(\left(\frac{1}{2}\mu|w|^2+q\right)w\right)&\leq 0
\end{align}
on $\mathscr D$, provided that \eqref{eq:euler_only} and thus also \eqref{eq:main0g} already hold true.
In fact not only the sign is shared by the two dissipation measures, there is a one-to-one translation in the sense that
\begin{align}\begin{split}\label{eq:translation_of_dissipation_measures}
\partial_t \left(\frac{1}{2}\rho|v|^2+\rho g x_n\right)+\divv& \left(\left(\frac{1}{2}\rho|v|^2+\rho g x_n+p\right)v\right)\\&=\partial_t \left(\frac{1}{2}\mu|w|^2\right)+\divv \left(\left(\frac{1}{2}\mu|w|^2+q\right)w\right)
\end{split}
\end{align}
in an appropriate distributional sense,
with the left-hand side evaluated in $(x,t)$ coordinates and the right-hand side in $(y,t)$ coordinates, see Lemma \ref{lem:translation}.

\subsection{The associated differential inclusion}\label{sec:differential_inclusion}

We now recast equations \eqref{eq:main0g}, \eqref{eq:euler_0g_local_energy_inequality} as a differential inclusion and formulate a corresponding convex integration theorem for the existence of turbulently mixing solutions. That is, we rewrite \eqref{eq:main0g}, \eqref{eq:euler_0g_local_energy_inequality} as the linear system 
\begin{align}\label{eq:linear_system_with_p}
\begin{split}
\partial_t m+\divv \sigma +\nabla \left(\frac{2}{n}e+q\right)&=0,\\
\divv w&=0,\\
\partial_t \mu+\divv m &=0,\\
\partial_t e +\divv h &=\nu,
\end{split}
\end{align}
coupled with the following set of nonlinear pointwise constraints:
\begin{align}\label{eq:pointwise_constraints_when_introducing_diff_inclusion}
    m=\mu w,\quad h=(e+q)w,\quad \mu w\otimes w=\sigma +\frac{2}{n} e \id, \quad \mu\in\{\rho_-,\rho_+\}.
\end{align}
Here the tuple of functions $z=(\mu,w,m,h,\sigma,e)$ defined on $\mathscr{D}$ takes values in the space $Z:=\R\times\R^n\times\R^n\times\R^n\times\cS_0^{n\times n}\times\R$, where $\cS_0^{n\times n}$ denotes the space of traceless symmetric $n$-by-$n$ matrices. Note that taking the trace of the matrix equality in \eqref{eq:pointwise_constraints_when_introducing_diff_inclusion} implies
\[
e=\frac{1}{2}\mu\abs{w}^2.
\]
Moreover, $q:\mathscr{D}\rightarrow\R$ is the pressure function and $\nu$ is a measure that we require to be negative in the sense that $\nu[\varphi]\leq 0$ for all $\varphi\in\cC^\infty_c(\cup_{t\in[0,T)}\overline{\mathscr{D}(t)}\times\{t\})$.

It is easy to check that if $z$ solves \eqref{eq:linear_system_with_p}, \eqref{eq:pointwise_constraints_when_introducing_diff_inclusion} then $(\mu,w)$ solves \eqref{eq:main0g}, \eqref{eq:euler_0g_local_energy_inequality} with pressure $q$, dissipation measure $\nu$ and $\mu\in\set{\rho_-,\rho_+}$. Note that we have manually added the latter condition as a consistency requirement between the transport equation and the initial data of our interest which has two-phase type, cf. \eqref{eq:flat_initial_data}. Of course having a solution to \eqref{eq:main0g}, \eqref{eq:euler_0g_local_energy_inequality} with this additional property, one can easily write up a tuple $z$, as well as $q$ and $\nu$ for which \eqref{eq:linear_system_with_p}, \eqref{eq:pointwise_constraints_when_introducing_diff_inclusion} hold true. 

Also note that \eqref{eq:boundary_condition_w} implies the boundary conditions
\begin{align}\label{eq:new_boundary_conditions}
    (m-\mu gte_n)\cdot \vec{n}=0=(h-(e+q)gte_n)\cdot \vec{n},~ \left(\sigma+\frac{2}{n}e\id-m\otimes(gte_n)\right)\vec{n}=0
\end{align}
on $\partial\mathscr{D}(t)$, $t\in(0,T)$, which are contained distributionally in Definition \ref{def:weaksolslin}.

\subsection{Statement of results 1 -- Convex integration}\label{sec:statement_of_results}

For the formulation of our convex integration theorem we also define the open set $Z_0:=\set{z\in Z:\mu\in(\rho_-,\rho_+)}$ and $\tilde M:Z_0\rightarrow \cS^{n\times n}$,
\begin{align}\label{eq:matrix}
\tilde M(z):=\frac{\mu\rho_-\rho_+w\otimes w-\rho_-\rho_+(m\otimes w+w\otimes m)+(\rho_++\rho_--\mu)m\otimes m}{(\rho_+-\mu)(\mu-\rho_-)}-\sigma,
\end{align}
as well as $T_\pm:Z_0\rightarrow \R$,
\begin{align}\label{eq:traces}
    T_\pm (z)=\frac{\rho_\pm}{2}\frac{|m-\rho_\mp w|^2}{(\mu-\rho_\mp)^2}.
\end{align}
\begin{theorem}\label{thm:main_theorem} Let $\Omega$ be a bounded domain or a periodic channel, $\mu_0\in L^\infty(\Omega)$ with $\mu_0\in\set{\rho_-,\rho_+}$ almost everywhere, and $w_0\in L^\infty(\Omega;\R^n)$, $\divv w_0=0$ in $\Omega$, $w_0\cdot \vec{n}=0$ on $\partial \Omega$. Suppose that there exists an $L^\infty$ solution $z=(\mu,w,m,h,\sigma,e)$ of \eqref{eq:linear_system_with_p} with boundary conditions \eqref{eq:boundary_condition_w}, \eqref{eq:new_boundary_conditions}, initial data
\[
\mu(\cdot,0)=\mu_0,\quad m(\cdot,0)=\mu_0w_0,\quad e(\cdot,0)=\frac{1}{2}\mu_0\abs{w_0}^2,
\]
pressure $q$ and negative dissipation measure $\nu\in \left(\cC_c^\infty(\cup_{t\in[0,T)}\overline{\mathscr{D}(t)}\times\{t\})\right)^*$. Suppose further that to this solution there exists an open set $\mathscr{U}\subset\mathscr D$, as well as $\varepsilon>0$, such that $(\mu,w,m,h,\sigma,e)$ is continuous on $\mathscr{U}$, 
\begin{align}\label{eq:Tpm_and_q_bounded}
    q\in L
^\infty(\mathscr{U})\cap \cC^0(\mathscr{U}),\quad T_-(z),T_+(z)\in L^\infty(\mathscr{U}),
\end{align}
 and such that on $\mathscr{U}$ there holds
\begin{gather}\begin{gathered}\label{eq:sufficient_condition_for_Linfty_hull}
\rho_-<\mu<\rho_+,\\
    \lamax(\tilde M (z))+\varepsilon\abs{h-(e+q)w-\frac{T_+(z)-T_-(z)}{\rho_+-\rho_-}(m-\mu w)}< \frac{2}{n}e,\end{gathered}
\end{gather}
whereas \eqref{eq:pointwise_constraints_when_introducing_diff_inclusion} is required to hold almost everywhere in $\mathscr D\setminus\mathscr{U}$.
Then there exist infinitely many weak solutions $(\mu_{sol},w_{sol})\in L^\infty(\mathscr D)$ of \eqref{eq:main0g}, \eqref{eq:euler_0g_local_energy_inequality}, \eqref{eq:boundary_condition_w} with initial data $(\mu_0,w_0)$, pressure $q$ and dissipation measure $\nu$, i.e. there holds $\mu_{sol}\in\{\rho_-,\rho_+\}$ almost everywhere and
\begin{align*}
    \partial_t(\mu_{sol} w_{sol})+\text{div }(\mu_{sol} w_{sol}\otimes w_{sol}) + \nabla q &= 0,\\
\text{div } w_{sol}&=0,\\
\partial_t \mu_{sol} + \text{div }(\mu_{sol} w_{sol})&=0,\\
\partial_t \left(\frac{1}{2}\mu_{sol}|w_{sol}|^2\right)+\divv \left(\left(\frac{1}{2}\mu_{sol}|w_{sol}|^2+q\right)w_{sol}\right)&=\nu.
\end{align*}
On $\mathscr D\setminus \mathscr{U}$ the solutions $(\mu_{sol},w_{sol})$ coincide with $(\mu,w)$. Furthermore, among the infinitely many solutions one can find sequences $(\mu_k,w_k,m_k,h_k,\sigma_k,e_k)$, defined via \eqref{eq:pointwise_constraints_when_introducing_diff_inclusion} such that $(\mu_k,w_k,m_k,h_k,\sigma_k,e_k)\rightharpoonup (\mu,w,m,h,\sigma,e)$ weakly in $L^2(\mathscr D)$ as $k\rightarrow \infty$.
\end{theorem}

\begin{definition}
Any tuple $z$ as in Theorem \ref{thm:main_theorem} is called a subsolution and $\mathscr{U}$ the associated mixing zone, cf. also Definition \ref{def:weaksolslin}.
\end{definition}
As it is common in convex integration, further attributes such as for instance the ``turbulent mixing at every time slice'' property  (c.f. e.g. \cite{Castro_Faraco_Mengual_degraded}) can be added to the solutions $(\mu_{sol},w_{sol})$. Regardless, Theorem \ref{thm:main_theorem} shows that the subsolution determines the crucial features of the induced solutions such as the extent of the mixing zone, but also the local dissipation measure $\nu$ and the pressure $q$. For instance, concerning the total energy of the induced solutions there holds
\begin{align}\label{eq:total_energy_of_sols}
    E(t)=\int_{\Omega+\frac{1}{2}gt^2}e(y,t)-m_n(y,t)gt+\mu(y,t) gy_n\:dy,~t\in[0,T),
\end{align}
provided that $\nu\in L^1(\Omega\times(0,T))$. This relation is stated and proven in the appendix, see Lemma \ref{lem:total_energy}.

Conditions \eqref{eq:Tpm_and_q_bounded}, \eqref{eq:sufficient_condition_for_Linfty_hull} are sufficient conditions for inducing weak solutions via convex integration in the Tartar framework, or in other words sufficient conditions for being in the interior of the convex hull of a compact subset of the set defined through the constraints \eqref{eq:pointwise_constraints_when_introducing_diff_inclusion}. In the other direction we have the following statement.
\begin{proposition}\label{prop:weak_limits}
Suppose $(\mu_k,w_k)\in L^\infty(\mathscr{D})\times L^3(\mathscr{D};\R^n)$, $k\in\N$ is a sequence
of solutions to \eqref{eq:main0g}, \eqref{eq:euler_0g_local_energy_inequality} with pressures $q_k$, dissipation measures $\nu_k$ and $\mu_k\in\set{\mu_-,\mu_+}$ pointwise almost everywhere. Define the associated tuple $z_k$ via \eqref{eq:pointwise_constraints_when_introducing_diff_inclusion}. If $z_k$ and $q_k$ converge weakly in $L^1(\mathscr{D})$ to an integrable tuple $\bar{z}$ and $\bar{q}\in L^1(\mathscr{D})$, and if $\nu_k$ converges weakly to another measure $\bar{\nu}$, then $\bar{z}$ solves \eqref{eq:linear_system_with_p} with pressure $\bar{q}$, and dissipation $\bar{\nu}$. Moreover, $\bar{z}$ takes pointwise almost everywhere values in the closure of the set specified through \eqref{eq:sufficient_condition_for_Linfty_hull} for parameter $\varepsilon=0$. If $\norm{w_k}_{L^\infty(\mathscr{D};\R^2)}$ is uniformly bounded along the sequence, then we in addition obtain $T_-(\bar{z}),T_+(\bar{z})\in L^\infty(\mathscr{D})$ with 
\begin{align*}
    \norm{T_\pm(\bar{z})}_{L^\infty(\mathscr{D})}\leq \frac{\rho_+}{2}\sup_k\norm{w_k}_{L^\infty(\mathscr{D};\R^2)}^2.
\end{align*}
\end{proposition}
This proposition is a consequence of the full characterization of the convex hull associated with \eqref{eq:pointwise_constraints_when_introducing_diff_inclusion} in the unbounded case, see Proposition \ref{prop:hull}, and the convexity of sublevelsets $\set{T_\pm\leq e}$ known from \cite[Lemma 4.4]{GKSz}. Note that the latter convexity only holds true when $e$ is considered as a fixed constant and not as a variable like it is done in the present paper.

\subsection{Statement of results 2 -- Subsolutions}\label{sec:statement_of_results_2}

Concerning the construction of subsolutions which induce turbulently mixing solutions by means of Theorem \ref{thm:main_theorem}, we focus on the flat unstable Rayleigh-Taylor configuration \eqref{eq:flat_initial_data} in a periodic channel $\Omega=\T^{n-1}\times (-L,L)$ with $\T$ having periodicity $1$ and $L>0$. In that context it is legitimate to understand subsolutions as horizontally averaged solutions. We therefore focus on subsolutions $z$ with associated pressure $q$ and dissipation measure $\nu$ with
\begin{gather}\label{eq:1Dsub1}
    z(y,t)=z(y_n,t),\quad q(y,t)=q(y_n,t),\quad \nu(y,t)=\nu(y_n,t),\\
    w(y_n,t)=gte_n,\label{eq:1Dsub2}\\
    m_j(y,t)=0,\quad h_j(y,t)=0,~j=1,\ldots,n-1.\label{eq:1Dsub3}
\end{gather}

Note that $w_n(y_n,t)=gt$ in \eqref{eq:1Dsub2} is in fact enforced by \eqref{eq:1Dsub1}, the incompressibility of $w$ and \eqref{eq:boundary_condition_w}. Vanishing of the remaining components, as well as conditions \eqref{eq:1Dsub3}, means that there is no average velocity, momentum, energy flux resp., in horizontal directions.

Moreover, for any subsolution satisfying \eqref{eq:1Dsub1}-\eqref{eq:1Dsub3}, one observes that the matrix $\tilde {M}(z)$ which gives a lower bound on the local energy density $e$ through its maximal eigenvalue, cf. \eqref{eq:sufficient_condition_for_Linfty_hull}, can be written as 
\begin{align}\label{eq:form_of_tilde_M_in_subs}
    \tilde{M}(z)=f(\mu,m_n,t)e_n\otimes e_n-\sigma,
\end{align}
with $f$ given by
\begin{align}\label{eq:f_in_subs}
    f(\mu,m_n,t):=\frac{\mu\rho_-\rho_+g^2t^2-2\rho_-\rho_+ m_n gt+(\rho_++\rho_--\mu)m_n^2}{(\rho_+-\mu)(\mu-\rho_-)}.
\end{align}
Looking at \eqref{eq:form_of_tilde_M_in_subs}, there is a canonical choice for $\sigma(y,t)\in\cS^{n\times n}_0$ turning $\tilde{M}(z)$ into a diagonal matrix and thus minimizing its maximal eigenvalue $\lamax(\tilde{M}(z))$. Depending on the considered dimension $n=2$ or $n=3$, this choice is
\begin{align}\label{eq:choice2}
    \sigma=\frac{1}{2}\begin{pmatrix}
    -f & 0\\
    0 &f
    \end{pmatrix},\quad\sigma=\frac{1}{3}\begin{pmatrix}
    -f & 0 &0\\
    0&-f&0\\
    0&0&2f
    \end{pmatrix}.
\end{align}
Here we focus on subsolutions with this canonical choice of $\sigma$.
\begin{definition}\label{def:1Dsubs} Consider initial data \eqref{eq:flat_initial_data} and $e_0(x)=0$ in the periodic channel.
A subsolution $z$ with initial data $(\rho_0,v_0=0,e_0=0)$, pressure $q$ and dissipation measure $\nu$ is called a one-dimensional subsolution provided properties \eqref{eq:1Dsub1}--\eqref{eq:choice2} hold true.
\end{definition}

There is a subtle difference when considering similar subsolutions in a box $\Omega=(0,1)^{n-1}\times(-L,L)$ instead of the periodic channel. Indeed, conditions \eqref{eq:boundary_condition_w}, \eqref{eq:new_boundary_conditions}, \eqref{eq:1Dsub1} automatically imply \eqref{eq:1Dsub2}, \eqref{eq:1Dsub3}, but also
\begin{align*}
    \sigma=\frac{1}{2}\begin{pmatrix}
    -e & 0\\
    0 &e
    \end{pmatrix},\quad\sigma=\frac{1}{3}\begin{pmatrix}
    -2e & 0 &0\\
    0&-2e&0\\
    0&0&4e
    \end{pmatrix}.
\end{align*}
Thus $\lamax(\tilde{M}(z))\geq \frac{2}{n}e$, which means that \eqref{eq:sufficient_condition_for_Linfty_hull} can at most hold with equality when 
\begin{gather}
    e=\frac{1}{2}f,\quad 
    h=(e+q)w+\frac{T_+(z)-T_-(z)}{\rho_+-\rho_-}(m-\mu w).\label{eq:nonstrictsubs}
\end{gather}
\begin{definition}
A one-dimensional non-strict subsolution is a one-dimensional subsolution except that \eqref{eq:sufficient_condition_for_Linfty_hull} is replaced by \eqref{eq:nonstrictsubs}. 
\end{definition}

Our convex integration theorem does not apply to non-strict subsolutions. Nonetheless, we mainly focus on them, and then show that one can perturb a non-strict subsolution to a strict one when working on $\Omega=\T^{n-1}\times(-L,L)$.

\begin{proposition}\label{prop:subsgen} Let $\Omega=\T^{n-1}\times(-L,L)$ or $\Omega=(0,1)^{n-1}\times(-L,L)$ and let us define 
\[
a(r):=\frac{(\rho_+-r)(r-\rho_-)}{\rho_++\rho_--r},\quad r\in[\rho_-,\rho_+].
\]
For any $\lambda\in(0,1/2]$,
choosing 
\begin{align}
    m=m(t,\mu)= gt (\mu-\lambda a(\mu)) e_n,\label{eq:m_in_subsol}
\end{align}
and taking $\mu$ to be the unique entropy solution of the Riemann problem 
\begin{align*}
    \partial_t\mu+\partial_{y_n}m(t,\mu)=0,\quad \mu(\cdot,0)=\rho_0,
\end{align*}
gives rise to a one-dimensional non-strict subsolution $z$ with pressure $q$ and dissipation measure $\nu$ given by
\begin{align*}
    \nu&=(1-2\lambda)g(m_n-\mu gt)= \lambda(2\lambda-1)g^2 t a(\mu)\leq 0.
\end{align*}
The associated total energy defined through the right-hand side of \eqref{eq:total_energy_of_sols} is given by
\begin{align}\label{eq:total_enerrgy_in_subs_prop}
    E_{tot}(t)=\int_{\Omega}\rho_0(y)gy_n\:dy+\frac{(\rho_+-\rho_-)^3}{12\rho_+\rho_-}\left(\lambda-\frac{1}{2}\right)\lambda^2g^3t^4.
\end{align}
All involved quantities $z$, $q$ and $\nu$ are continuous on all of $\overline{\mathscr{D}}$ except $\mu$ at $\set{x_n=t=0}$.
\end{proposition}

As indicated in Proposition \ref{prop:lambda_1_3}, the parameter $\lambda =\frac{1}{3}$ is special. We give the following more detailed version of said proposition.
\begin{proposition}\label{prop:energies_of_subsolution_sec2}
Let $z$ be a one-dimensional potentially non-strict subsolution with $m_n$ given through a constitutive relation of the form 
\begin{align}\label{eq:general_m_with_F}
    m_n=gt(\mu +F(\mu))
\end{align}
for some uniformly convex $F\in\cC^2([\rho_-,\rho_+])$ vanishing at the endpoints $\rho_-,\rho_+$. If $\mu$ is taken as the unique entropy solution of the induced Riemann problem \eqref{eq:flat_initial_data}, \eqref{eq:linear_system_with_p}, then the total energy balance satisfies
\begin{align*}
    \frac{4}{g^3t^4}(E_{tot}(0)-E_{tot}(t))\leq  \int_{\rho_-}^{\rho_+}\frac{d}{dr}\left(a(r)^{-1}F(r)^2\right)F'(r)+\frac{1}{2}F'(r)^2\:dr=:\cI(F).
\end{align*}
This inequality becomes an equality for non-strict subsolutions. Moreover, there holds
\begin{align*}
    \cI(F)\leq \cI(-a/3)=\frac{(\rho_+-\rho_-)^3}{162\rho_+\rho_-}
\end{align*}
for all such $F$.
\end{proposition}
The ansatz \eqref{eq:general_m_with_F} with $F$ vanishing at the endpoints is motivated by compatibility at the boundary of the mixing zone, since outside, where $\mu\in\set{\rho_\pm}$, one has to have $m=\mu gte_n$.

Finally, we address the need for a strict subsolution in order to carry out convex integration.
\begin{lemma}\label{lem:strict_subs} Let $\Omega=\T^{n-1}\times(-L,L)$ and $z_\lambda$, $\lambda\in(0,1/2)$ be the non-strict subsolution of Proposition \ref{prop:subsgen} with associated quantities $q_\lambda$, $\nu_\lambda$. For all $\delta>0$ sufficiently small depending on $\lambda,\rho_-,\rho_+,g$, there exists a strict subsolution $z^\delta_\lambda$ with 
\begin{gather*}
    \mu^\delta_\lambda=\mu_\lambda,\quad m_\lambda^\delta=m_\lambda,\quad e^\delta_\lambda=e_\lambda+\delta t^2(\rho_+-\mu_\lambda)^2(\mu_\lambda-\rho_-)^2,
\end{gather*}
and dissipation measure $\nu^\delta_\lambda\leq 0$ satisfying
\begin{align*}
    \abs{\nu_\lambda^\delta- \nu_\lambda}\leq C\delta t(\rho_+-\mu_\lambda)(\mu_\lambda-\rho_-)
\end{align*}
for a constant $C=C(\lambda,g,\rho_-,\rho_+)>0$.
The parameter $\varepsilon>0$ in \eqref{eq:sufficient_condition_for_Linfty_hull} can be chosen to be $1$.
\end{lemma}

Now Theorem \ref{thm:introduction} follows from Lemma \ref{lem:strict_subs}, Proposition \ref{prop:subsgen}, Theorem \ref{thm:main_theorem} and \eqref{eq:total_energy_of_sols}.

\subsection{Related results}

The viewpoint of the Euler equations as a differential inclusion has been pioneered by the works of De Lellis and Sz\'ekelyhidi \cite{DeL-Sz-Annals,DeL-Sz-Adm}, which  enabled a machinery for the construction of wild solutions for various fluid equations serving as counterexamples or as a description of hydrodynamic instabilities for which the initial value problem is naturally ill-posed, see for instance the surveys \cite{Buckmaster-Vicol-Survey,DeL-Sz-Survey2,DeL-Sz-Survey}.

Concerning instabilities, the first example of an application has been given by Sz\'ekelyhidi \cite{Sz_KH} in the context of the Kelvin-Helmholtz instability. Since then, relaxations of fluid differential inclusions and the investigation of corresponding subsolutions have been addressed in various contexts, see \cite{GK-EE,Mengual_Sz_sheets} for further works on the Kelvin-Helmholtz instability, \cite{GHK_LAP,GK_Boussinesq,GKSz} for Rayleigh-Taylor, \cite{Arnaiz_Castro_Faraco,Castro_Cordoba_Faraco_invent,Castro_Faraco_G,Castro_Faraco_Mengual_degraded,Castro_Faraco_Mengual_turned,Cordoba_Faraco_Gancedo_lack_of_uniqueness,Foerster_Sz,Mengual_different_mobilities,Hitruhin_Lindberg_ipm,Noisette_Sz,Sz_ipm} for IPM and Saffman-Taylor, \cite{Gancedo_HT_Mengual_circular_filament} for vortex filaments, \cite{G_optimal_mixing} for optimal mixing, \cite{Faraco_Lindberg_Sz_bounded_mhd,Faraco_Lindberg_Sz_second_paper,Hitruhin_Lindberg_dynamo} for MHD and the kinematic dynamo equation. In addition, we would like to mention the recent work \cite{Enciso_etal_2024} which implements a Nash-type convex integration scheme in order to show the existence of admissible, H\"older continuous (for positive times) solutions emanating from the flat Kelvin-Helmholtz instability.

\subsubsection{Local energy inequality}
Except for \cite{GK-EE}, the solutions to the Euler and Boussinesq system obtained in the articles listed above are all only known to be weakly admissible in the sense that their total energy at every or almost every time does not exceed the initial total energy.
One of the main emphases in the present paper is to provide weak solutions that in addition satisfy the local energy inequality. It has been known since the work of De Lellis and Sz\'ekelyhidi \cite{DeL-Sz-Adm} that also this stronger admissibility criterion does not recover uniqueness of weak solutions, more precisely of $L^\infty$ solutions. Thereafter, convex integration techniques have been applied in various compressible Euler settings to disprove well-posedness of entropy solutions, see  \cite{Akramov_Wiedemann2021,Brezina_etal_2018,Chiodaroli2014,Chiodaroli_DeLellis_Kreml2015,Chiodaroli_Feireisl_dense,Chiodaroli_Kreml2014,Chiodaroli_Kreml2018,Chiodaroli_smooth2021,Feireisl_max_dissip,Klingenberg_Markfelder_conservative,Markfelder_book,Markfelder_Klingenberg2018} for some examples.

In addition, counterexamples for the homogeneous Euler equations with higher regularity have been provided by Isett \cite{Isett_dissip}, De Lellis and Kwon \cite{DeLellis_Kwon_dissip}, as well as Giri, Kwon and Novack \cite{Giri_Kwon_Novack_2,Giri_Kwon_Novack_1}, where the $L^3$-based strong Onsager conjecture has been reached. H\"older regular counterexamples to compressible Euler systems have also been obtained, see \cite{Giri_Kwon_compr,Mao_Qu_compr}.

Besides treating the inhomogeneous Euler equations instead of the homogeneous ones, the focus of the present paper in contrast lies not in reaching a certain regularity of the obtained solutions, but on the description of the whole set of subsolutions or Euler-Reynolds flows inducing weak solutions via convex integration in the $L^\infty$-framework. In other words, on the explicit relaxation of the Euler equations when seen as a differential inclusion. An understanding of this relaxation is important for an application to hydrodynamic instabilities. 

In \cite{GK-EE} the authors have addressed this question for the homogeneous Euler equations by including the local energy inequality into the original differential inclusion of \cite{DeL-Sz-Annals,DeL-Sz-Adm}. This way it could be shown that the subsolutions of Sz\'ekelyhidi \cite{Sz_KH} for the Kelvin-Helmholtz instability give rise to turbulent solutions respecting local energy dissipation. 
A similar approach has been taken by Markfelder \cite{Markfelder_EE} for the barotropic compressible Euler equations which allowed him to show the existence of Riemann data for which the local maximal dissipation criterion does not favour the self-similar solution.

In comparison to \cite{GK-EE} where $\rho\equiv 1$, and \cite{Markfelder_EE} where $\rho$ is treated as a given function, the presence of a variable density which is forced to take the values $\rho_\pm$ requires significant care when computing the relaxation. Note for instance that the energy flux $h=\left(\frac{1}{2}\mu\abs{w}^2+q\right)w$ is an order $4$ polynomial. In fact we deviate from the strategy of \cite{GK-EE} when computing the convex hull, which as a by-product also allows us to improve the results for the homogeneous Euler case, see Remark \ref{rem:EE}.

\subsubsection{Inhomogeneous Euler and Boussinesq}

The inhomogeneous Euler equations without local energy inequality, as well as their Boussinesq approximation, have previously been treated as a differential inclusion in \cite{GKSz} by the authors and Sz\'ekelyhidi, and \cite{GK_Boussinesq} by the authors. The relaxations obtained in the two papers have been applied to the Rayleigh-Taylor instability. We will give a short comparison, but aside from the Rayleigh-Taylor instability first we would like to mention further convex integration results for the inhomogeneous Euler and Boussinesq systems providing non-uniqueness results.  

Relying on convex integration for the homogeneous Euler equations, \cite{Chiodaroli_Michalek} shows non-uniqueness of weak solutions with non-increasing total energy for the Boussinesq system with initial data $\rho_0\in L^\infty\cap\cC^2$. Recently also Nash-type convex integration has been implemented for the inhomogeneous Euler equations (without gravity on $\T^3$) \cite{Giri_Ujjwal_Euler} and the Boussinesq system (with gravity on $\T^2$) \cite{Ujjwal_Boussinesq}, in order to provide wild solutions with non-trivial H\"older continuous density and velocity.

Coming back to a comparison with \cite{GK_Boussinesq,GKSz}, the results obtained in the present paper generalize those of \cite{GKSz} on several levels. On one hand, as mentioned before, our differential inclusion also accounts for local energy dissipation, hence all solutions obtained through convex integration satisfy this property, while a priori this was not guaranteed by the analysis in \cite{GKSz}. On the other hand, the latter paper only achieved a construction of strict subsolutions when the Atwood number was in the ultra-high range ($\rho_+>(\frac{4+2\sqrt{10}}{3})^2\rho_-$), and the subsolution did not arise from any rigorous optimization principle, while our current construction works for any range of Atwood numbers and provides a potential selection criterion for subsolutions by maximizing total energy dissipation. Though one should note that in both papers the density component of the subsolution arose as the entropy solution of an appropriate conservation law, both yielding turbulent mixing zones which grow proportionally to $gt^2$. In particular, the mixing zone of \cite{GKSz} was represented by the set $\frac{2x_n}{gt^2}\in \left(\sqrt{q}-1,\sqrt{q^{-1}}-1\right)$, $q:=\frac{\rho_-}{\rho_+}$, while in our case it is given by $\frac{2x_n}{gt^2}\in \left(\lambda(q-1),\lambda(q^{-1}-1)\right)$. Taking the maximally dissipating one, i.e. $\lambda=\frac{1}{3}$, it is easy to see that whenever $q^{-1}\geq 4$ our mixing zone is the larger one, and this is clearly the case in the ultra-high Atwood range for which the construction of subsolutions via the method in \cite{GKSz} was possible.

In comparison to \cite{GK_Boussinesq}, where the Boussinesq system, which is an approximation of the Euler equations for small Atwood numbers, is considered, we would like to note that taking the limit $A\to 0^+$ in \eqref{eq:rho_bar_explicit}, one obtains
$$ \bar{s}_\lambda\left(\frac{\lambda gt^2}{2}A\xi,t\right)= \frac{\xi}{2}+O(A),\quad \xi\in \left(-2+O(A),2+O(A)\right),$$
which for $\lambda=\frac{1}{3}$ corresponds exactly to the linear profile for the normalized density obtained in \cite{GK_Boussinesq}. However, one should note that in said paper the authors maximized total energy dissipation at initial time among all possible self-similar densities $\rho(x,t)$ as a selection criterion, whereas here we maximize the total energy dissipation with respect to the constitutive law relating $\rho$ and $m$. The correspondence as $A\rightarrow 0$ further shows that the current paper yields a larger framework which both unites and generalizes  the results of \cite{GK_Boussinesq,GKSz}.

\subsubsection{Selection of subsolutions}\label{sec:selection}

A very interesting problem in addressing hydrodynamic instabilities via differential inclusions and convex integration is the choice or selection of appropriate subsolutions, i.e. appropriate coarse grained or macroscopic behaviour. Indeed, the relaxation consists of a linear equation coupled with pointwise convex inequalities naturally leaving room for more than one construction, each of them, in particular for non-flat initial interfaces, being far from trivial, see for instance \cite{Arnaiz_Castro_Faraco,Castro_Cordoba_Faraco_invent,Castro_Faraco_G,Foerster_Sz,Noisette_Sz} for different types of subsolutions emanating from the non-flat Saffman-Taylor instability.

One possible strategy for a selection, in particular for a flat initial interface, is to reduce the subsolution system to a particular hyperbolic conservation law for the coarse grained density and to consider the corresponding entropy solution. This has been done in \cite{GKSz,Sz_KH,Sz_ipm} for the flat Rayleigh-Taylor, Kelvin-Helmholtz, Saffman-Taylor instabilities respectively.

In the present paper we also pursue this strategy, but, as a novelty, leave the constitutive relation between the coarse grained density and coarse grained momentum (=mass flux) open. With this relation determining the conservation law, we then investigate the total energy associated with the induced entropy solution. As a second criterion, besides imposing such a constitutive relation in the first place, we then maximize the total (anomalous) energy dissipation leading to the subsolution stated in Theorem \ref{thm:introduction} with $\lambda=\frac{1}{3}$ and Proposition \ref{prop:lambda_1_3}.

This second criterion is broadly speaking motivated by the general search for extremal subsolutions which are potentially preferred by nature. The criterion of maximizing total energy dissipation within a certain class of subsolutions is motivated by Dafermos' entropy rate admissibility criterion \cite{Dafermos1973}, and has been investigated in the context of convex integration solutions for compressible Euler in \cite{Chiodaroli_Kreml2014,Feireisl_max_dissip}. As a selection criterion it has only been applied at initial time before, see \cite{GK_Boussinesq,Mengual_Sz_sheets}.
In the context of IPM and the Saffman-Taylor instability, maximal dissipation of potential energy has been used in \cite{Castro_Faraco_G}. It turned out that this selection in fact coincides with the relaxation of Otto \cite{Otto1999} based on the gradient flow structure of IPM.
A similar correspondence between two relaxations has been found in \cite{GHK_LAP} when imposing the least action principle to Boussinesq subsolutions for the Rayleigh-Taylor instability and comparing the resulting variational problem with the direct relaxation of the least action principle of Brenier \cite{Brenier1989,Brenier2018} based on generalized incompressible flows.

In the context of admissibility criteria we would also like to mention the least action admissibility principle by Gimperlein et al. recently introduced in \cite{Gimperlein_etal_1} and refined in \cite{Gimperlein_etal_2}, after an objection of Markfelder and Pellhammer \cite{Markfelder_Pellhammer}. This criterion however is applied in the context of solutions, not subsolutions, to a Riemann problem of the compressible Euler equations favoring the 2-shock solution over entropy solutions obtained via convex integration. Probing the least action admissibility principle within the class of Rayleigh-Taylor subsolutions considered in the present article, it turns out that the stationary solution $\rho(x,t)=\rho_0(x)$, i.e. no proper subsolution and thus also here no convex integration solution, is preferred by the criterion of \cite{Gimperlein_etal_1,Gimperlein_etal_2}, see Remark \ref{rem:laap}. In this class of subsolutions the least action admissibility principle therefore does not pick a suitable macroscopic evolution.

\subsubsection{Connection to scale invariant bounds}

As mentioned before, in the family of subsolutions stated in Theorem \ref{thm:introduction} we find for $\lambda =\frac{1}{2}$ precisely the background profile realizing sharpness in the scale invariant bounds of \cite{Kalinin_Menon_Wu}. In the latter article a more regular system including the effect of miscibility is treated, and uniform bounds for bulk quantities such as potential energy or  mixing entropy are obtained using energy conservation as a crucial ingredient. Note that within our family of subsolutions the one corresponding to $\lambda=\frac{1}{2}$ indeed is the conservative one. Similar bounds are also available for the Saffman-Taylor instability \cite{Menon_Otto}, and more recently the Kelvin-Helmholtz instability \cite{Kalinin_Menon_Wu_2}. In the latter work a similar connection between a subsolution of \cite{Sz_KH}, again the energy conservative one, and the optimal background profile has been found.

\section{Convex integration via the Tartar framework}\label{sec:convex_integration}
Recall that $\mathscr{D}$ is the accelerated spacetime set $\cup_{t\in(0,T)}(\Omega+\frac{1}{2}gt^2e_n)\times\{t\}$ and that $Z=\R\times\R^n\times\R^n\times\R^n\times \cS_0^{n\times n}\times\R$.
Let $q:\mathscr{D}\rightarrow \R$ and $\nu\in \cC^\infty_c((\Omega\times\{0\})\times \mathscr{D})^*$ be a  pressure function and a dissipation measure. For convex integration itself it is not required that $\nu$ be negative, if it is not, then the only downside is that the induced solutions fail to satisfy the energy inequalities \eqref{eq:euler_0g_local_energy_inequality} or \eqref{eq:euler_local_energy_inequality}.

As described in Section \ref{sec:differential_inclusion}, we are interested in the differential inclusion
\begin{align}\label{eq:abstract_differential_inclusion}
    \cA[z]=b,\quad z(y,t)\in K_{(y,t)}~\text{for almost every }(y,t)\in \mathscr{D},
\end{align}
where $z:\mathscr{D}\rightarrow Z$ is a tuple of functions $z=(\mu,w,m,h,\sigma,e)$, the linear differential operator $\cA$ and the right-hand side $b$ are defined by
\begin{align*}
    \cA[z]=\begin{pmatrix}
    \partial_t m+\divv \sigma + \frac{2}{n}\nabla e\\
\divv w\\
\partial_t \mu+\divv m\\
\partial_t e +\divv h 
    \end{pmatrix},\quad b=\begin{pmatrix}
    -\nabla q\\0\\0\\\nu
    \end{pmatrix},
\end{align*}
such that $\cA[z]=b$ is precisely \eqref{eq:linear_system_with_p}. Moreover, for the given pressure function $q$ the family of sets $K_{(y,t)}$ are defined through \eqref{eq:pointwise_constraints_when_introducing_diff_inclusion}, that is 
\begin{align}\label{eq:nonlinear_constraints}
\begin{split}
&K_{(y,t)}:=\left\{z\in Z:\ m=\mu w,\ h=(e+q(y,t))w, \phantom{\frac{1}{2}}\right.
\\
&\hspace{165pt}\left. \mu w\otimes w=\sigma +\frac{2}{n} e \id, \ \mu\in\{\rho_-,\rho_+\}\right\}.
\end{split}
\end{align}

\subsection{Localized plane waves}\label{sec:waves}
As perturbations we will use plane-wave like solutions to \eqref{eq:linear_system_with_p} which actually do not perturb the pressure or the dissipation rate. That is, we consider the system
\begin{align}\label{eq:linear_system}
\cA[\bar z]=0.
\end{align}
Clearly if $z$ is a weak solution of \eqref{eq:linear_system_with_p}, i.e. $\cA[z]=b$, and $\bar z\in\cC_c^\infty(\mathscr D)$ solves \eqref{eq:linear_system}, then $\cA[z+\bar{z}]=b$ and $z+\bar{z}$ has the same initial and boundary data as $z$.

The wave cone associated with the operator $\cA$ reads
\begin{equation}\label{eq:wave_cone}
\Lambda=\set{\bar{z}\in Z:\ker \begin{pmatrix}
\bar{\sigma}+\frac{2}{n}\bar{e}\id & \bar{m}\\
\bar{w}^T & 0\\
\bar{m}^T & \bar{\mu}\\
\bar h^T & \bar e
\end{pmatrix} \neq \{0\},\quad (\bar{m},\bar \mu,\bar{e})\neq0}.
\end{equation}
For each direction in the cone $\Lambda$ one can construct plane wave solutions of \eqref{eq:linear_system} oscillating in that direction. These plane waves are localized via potentials in Lemma \ref{lem:locpw}. In the definition of $\Lambda$ we have manually excluded the case where $\bar{m}$, $\bar{\mu}$, $\bar{e}$ simultaneously vanish as this case allows for pure time oscillations that are harder to localize.

\begin{lemma}\label{lem:locpw}
There exists $C_0>0$ such that for any $\bar{z}\in\Lambda$, there exists a sequence
$z_N\in \cC_c^\infty(B_1(0);Z)$ which
solves \eqref{eq:linear_system} and satisfies
\begin{itemize}
\item[(i)] $d(z_N,[-\bar{z},\bar{z}])\to 0$ uniformly,
\item[(ii)] $z_N\rightharpoonup 0$ in $L^2(B_1(0);Z)$,
\item[(iii)] $\int\int |z_N|^2\, dx \, dt\geq C_0|\bar{z}|^2,$
\end{itemize}
where $B_1(0)$ denotes the unit ball in $\R^n\times\R$.
\end{lemma}

\begin{proof}
We present most of the proof for the case $n=2$, then explain the differences that need to be adapted for $n=3$ at the end.

It is easy to see that for any smooth functions $\psi:\mathbb{R}^{2+1}\to\mathbb{R}$, $\phi:\mathbb{R}^{2+1}\to \mathcal S_0^{2\times 2}$, $\pi:\mathbb{R}^{2+1}\to\mathbb{R}^2$, setting $D(\phi,\psi,\pi)=(\mu,w,m,h,\sigma,e)$ with
\begin{gather*}
\mu=\divv\divv(\phi+\divv \pi\id),\quad
w=\nabla^\perp\psi,\quad
m=-\partial_t \divv (\phi+\divv \pi\id),\\
e=\divv \partial_{tt}\pi,\quad
\sigma=\partial_{tt} \phi,\quad
h = -\partial_{ttt}\pi,
\end{gather*}
yields a solution of \eqref{eq:linear_system} in the form of $D(\phi,\psi,\pi)$.

It follows from \eqref{eq:wave_cone} that for any $\bar{z}\in\Lambda$
there exists 
\begin{align}\label{eq:xc}
0\neq(\xi,c)\in\ker \begin{pmatrix}
\bar{\sigma} +\bar{e}\id & \bar{m}\\
\bar{w}^T & 0\\
\bar{m}^T & \bar{\mu}\\
\bar h^T & \bar e
\end{pmatrix}.
\end{align}
We consider a smooth function $S:\mathbb{R}\to\mathbb{R}$, $N\geq 1$, and carry out the construction of potentials using the above ingredients, split into two cases: whether $c$ is zero or not.

\textbf{Case 1: $c\neq0$}

In this case it follows that there also holds $\xi\neq0$, since $\xi=0$ would imply $(\bar{m},\bar \mu,\bar{e})=0$. Furthermore, without loss of generality one may assume that $|\xi|=1$.

Let us define
\begin{align*}
\phi_N(x,t)&=\frac{1}{c^2}\bar{\sigma}\frac{1}{N^2}S'(N(\xi,c)\cdot(x,t)),\\
\psi_N(x,t)&=\xi^\perp\cdot\bar{w}\frac{1}{N}S''(N(\xi,c)\cdot(x,t)),\\
\pi_N(x,t)&=\frac{1}{ c^2}\bar e \xi\frac{1}{N^3}S(N(\xi,c)\cdot(x,t)).
\end{align*}
We claim that there holds
\begin{align}\label{eq:c1pot}
D(\phi_N,\psi_N,\pi_N)=(\bar\mu,\bar w,\bar m,(\bar h\cdot \xi)\xi,\bar \sigma,\bar e)S'''(N(\xi,c)\cdot(x,t)).
\end{align}
Indeed, on one hand, one has
$$\phi_N(x,t)+\divv \pi_N(x,t)\id=\frac{1}{c^2}(\bar{\sigma}+\bar{e}\id)\frac{1}{N^2}S'(N(\xi,c)\cdot(x,t)),$$
and using \eqref{eq:xc}, there holds
\begin{align*}
\divv\divv(\phi_N+\divv \pi_N\id)&=\frac{1}{c^2}\xi^T(\bar{\sigma}+ \bar{e}\id)\xi S'''(N(\xi,c)\cdot(x,t))\\&=\frac{1}{c^2}\xi^T(-c\bar{m})S'''(N(\xi,c)\cdot(x,t))=\bar{\mu}S'''(N(\xi,c)\cdot(x,t)),\\
\partial_t\divv (\phi_N+\divv \pi_N\id)&=\frac{1}{c}(\bar{\sigma}+ \bar{e}\id)\xi S'''(N(\xi,c)\cdot(x,t))=-\bar{m} S'''(N(\xi,c)\cdot(x,t)),\\
\partial_{tt}(\phi_N+\divv \pi_N\id)&=(\bar{\sigma}+ \bar{e}\id) S'''(N(\xi,c)\cdot(x,t)),\\
\nabla^{\perp}\psi_N&=(\xi^\perp\cdot\bar{w})\xi^{\perp}S'''(N(\xi,c)\cdot(x,t)) =\bar{w} S'''(N(\xi,c)\cdot(x,t)),\\
-\partial_{ttt}\pi_N&=-c\bar e\xi S'''(N(\xi,c)\cdot(x,t)) =(\bar h\cdot \xi) \xi S'''(N(\xi,c)\cdot(x,t)).
\end{align*}

We fix this final error in $h$ by observing that for any smooth function $\alpha:\mathbb{R}^{2+1}\to\mathbb{R}$, $\breve{D}(\alpha)=(0,0,0,\nabla^\perp\alpha,0,0)$ also solves \eqref{eq:linear_system}. Setting
$$\alpha_N(x,t)=(\bar h\cdot \xi^\perp)\frac{1}{N}S''(N(\xi,c)\cdot(x,t)),$$
we get
$$\nabla^\perp\alpha_N=(\bar h\cdot \xi^\perp)\xi^\perp S'''(N(\xi,c)\cdot(x,t)),$$
and since $\bar h =(\bar h\cdot \xi^\perp)\xi^\perp+(\bar h\cdot \xi)\xi$, we end up with
$$D(\phi_N,\psi_N,\pi_N)+\breve{D}(\alpha_N)=\bar z S'''(N(\xi,c)\cdot(x,t)).$$

One can then localize these potentials similarly
as in for instance \cite{Cordoba_Faraco_Gancedo_lack_of_uniqueness,DeL-Sz-Annals}, by fixing $S(\cdot)=-\sin(\cdot)$ and, for $\varepsilon>0$, considering $\chi_\varepsilon\in \cC_c^\infty(B_1(0))$ satisfying $|\chi_\varepsilon|\leq 1$ on $B_{1}(0)$, $\chi_\varepsilon=1$ on $B_{1-\varepsilon}(0)$. It is then easy to check that 
$z_N=D(\chi_\varepsilon(\phi_N,\psi_N,\pi_N))+\breve{D}(\chi_\varepsilon\alpha_N)$ satisfies desired properties in order to conclude the proof of the lemma in this first case.

\textbf{Case 2: $c=0$} 

This case closely follows the corresponding case in \cite{GKSz}. We have $\xi\neq0$, so we once more may also assume without loss of generality that $|\xi|=1$. Furthermore, \eqref{eq:xc} implies that there exist constants $k_1,k_2,k_3,k_4\in\mathbb{R}$ such that 
\begin{align}\label{eq:perps}
\bar{w}=k_1\xi^\perp,\quad
\bar{m}=k_2\xi^\perp,\quad
\bar{\sigma}+\bar{e}\id=k_3\xi^\perp\otimes\xi^\perp,\quad
\bar h = k_4 \xi^\perp
.\end{align}
Define
\begin{gather*}
\phi_N(x)=0,\quad
\psi_N(x)=\xi^\perp\cdot\bar{w}\frac{1}{N}S''(N\xi\cdot x),\quad
\pi_N(x)=\bar\mu \xi\frac{1}{N^3}S(N\xi\cdot x),
\end{gather*}
from where similar calculations as in Case 1 yield that
\begin{align}\label{eq:c21pot}
D(\phi_{N},\psi_{N},\pi_N)=(\bar{\mu},\bar{w},0,0,0,0)S'''(N\xi\cdot x).
\end{align}

One can then treat the remaining terms $(\bar{m},\bar h,\bar{\sigma},\bar{e})$
as done simply for the homogeneous Euler equations, for instance in \cite[Remark 2]{DeL-Sz-Annals}.
Indeed,
it follows via a straightforward direct calculation that for any smooth function $\omega:\mathbb{R}^{2+1}\to\mathbb{R}^{2+1}$, defining $W=\curl_{(x,t)}\omega$ and $\tilde{D}(\omega)=(0,0,m,0,\sigma,e)$ by
\begin{align*}
m=-\frac{1}{2}\nabla^\perp W_3,\quad
\sigma+e\id=\begin{pmatrix}
\partial_2 W_1 & \frac{1}{2}(\partial_2W_2-\partial_1W_1)\\
\frac{1}{2}(\partial_2W_2-\partial_1W_1) & -\partial_1W_2
\end{pmatrix}
\end{align*}
implies that $\tilde{D}(\omega)$ solves \eqref{eq:linear_system}.

Then one may take $\omega$ to be of the form
$$\omega_N(x)=(a,b,a)\frac{1}{N^2} S'(N \xi\cdot x),$$
for some constants $a,b\in\mathbb{R}$, with $S$ as before, to obtain that
\begin{gather*}
\begin{pmatrix}
\partial_2 W_1 & \frac{1}{2}(\partial_2W_2-\partial_1W_1)\\
\frac{1}{2}(\partial_2W_2-\partial_1W_1) & -\partial_1W_2
\end{pmatrix}=a\xi^\perp\otimes\xi^\perp S'''(N \xi\cdot x),\\
\nabla^\perp W_3=(\xi_1b-\xi_2a)\xi^\perp S'''(N \xi\cdot x).
\end{gather*}
If $\xi_1\neq 0$, one may set $a=k_3$, $b=\frac{-2k_2+k_3\xi_2}{\xi_1}$, such that \eqref{eq:perps} yields
\begin{align*}
\tilde{D}(\omega_{N})=(0,0,\bar{m},0,\bar{\sigma},\bar{e})S'''(N\xi\cdot x).
\end{align*}
Then, using \eqref{eq:c21pot}
as well as $\breve D$ as defined in Case 1 and
$$\alpha_n(x)= k_4\frac{1}{N}S''(N\xi\cdot x),$$ 
one obtains
\begin{align*}
D(\phi_{N},\psi_{N},\pi_N)+\tilde{D}(\omega_{N})+\breve D(\alpha_n)=\bar{z} S'''(N\xi\cdot x).
\end{align*}
The localization is then done similarly to the previous case, via $z_N=D(\chi_\varepsilon(\phi_{N},\psi_{N},\pi_N))+\tilde{D}(\chi_\varepsilon\omega_{N})+\breve{D}(\chi_\varepsilon\alpha_{N}).$

If $\xi_1=0$, then the choice $a=k_3$ yields
\begin{align*}
\tilde{D}(\omega_{N})=\left(0,0,\frac{k_3}{2}\xi_2\xi^\perp,0,\bar{\sigma},\bar{e}\right)S'''(N\xi\cdot x).
\end{align*}
We may then introduce a final corrective potential by noting that for any smooth function $\theta:\mathbb{R}^{2+1}\to\mathbb{R}$, $\hat{D}(\theta)=(0,0,\nabla^\perp\theta,0,0,0)$ also solves \eqref{eq:linear_system}. 
Hence, defining
$$\theta_N(x)=\left(k_2-\xi_2\frac{k_3}{2}\right)\frac{1}{N}S''(N\xi\cdot x),$$
one obtains that
$$\nabla^\perp\theta_N(x)=\left(k_2-\xi_2\frac{k_3}{2}\right)\xi^\perp S'''(N\xi\cdot x),$$
and  using \eqref{eq:perps}, this further yields
\begin{align*}
D(\phi_{N},\psi_{N},\pi_N)+\tilde{D}(\omega_{N})+\breve D(\alpha_N)+\hat{D}(\theta_N)=\bar{z} S'''(N\xi\cdot x).
\end{align*}
The localization is once more done in a similar manner to the previous cases. This concludes the proof of the lemma for $n=2$.

\textbf{Dimension three.}
In the case $n=3$ we have that for any smooth functions $\psi:\mathbb{R}^{n+1}\to\mathbb{R}^n$, $\phi:\mathbb{R}^{n+1}\to \mathcal S_0^{n\times n}$, $\pi:\mathbb{R}^{n+1}\to\mathbb{R}^n$, setting $D(\phi,\psi,\pi)=(\mu,w,m,h,\sigma,e)$ with 
\begin{gather*}
\mu=\divv\divv\left(\phi+\frac{2}{n}\divv \pi\id\right),\quad
w=\nabla\times \psi,\quad
m=-\partial_t \divv \left(\phi+\frac{2}{n}\divv \pi\id\right),\\
e=\divv \partial_{tt}\pi,\quad
\sigma=\partial_{tt} \phi,\quad
h = -\partial_{ttt}\pi,
\end{gather*}
will solve \eqref{eq:linear_system}. Let us explain how using $\nabla\times$ instead of $\nabla^\perp$ affects the construction, the rest of the adaptations are straightforward in Case 1.

First, to achieve $\bar w$, we set
$$\psi_N(x,t)=(\bar w\times\xi)\frac{1}{N}S''(N(\xi,c)\cdot(x,t)).$$

Then, we fix the error in $\bar h$ in the following way. Complete $\xi$ to an orthonormal basis $(\zeta_1,\zeta_2,\zeta_3)$ of $\R^3$ with $\zeta_1=\xi$. We then have $\bar h=(\bar h \cdot \zeta_1)\zeta_1+(\bar h \cdot \zeta_2)\zeta_2+(\bar h \cdot \zeta_3)\zeta_3$, and
\begin{align*}
\zeta_3=\xi\times \zeta_2,\quad \zeta_2=-\xi\times \zeta_3.
\end{align*}
Since for any smooth $\alpha:\R^{3+1}\to\R^3$, $(0,0,0,\nabla\times\alpha,0,0)$ solves \eqref{eq:linear_system}, we may use
\begin{align*}
\alpha_N(x,t)=((\bar h \cdot \zeta_3)\zeta_2-(\bar h \cdot \zeta_2)\zeta_3)\frac{1}{N}S''(N(\xi,c)\cdot(x,t))
\end{align*}
to achieve the desired correction of $\bar h$.

Finally, Case 2 can be treated in higher dimensions as follows. Instead of $\tilde D$, we use potentials analogous to the construction in \cite{DeL-Sz-Annals}. Instead of \eqref{eq:perps}, we only have that $\bar w$, $\bar m$, $\bar h$ are in the plane perpendicular to $\xi$, however we can still achieve these directions similarly to our construction above for Case 1, $n=3$.
\end{proof}

\subsection{The unbounded and bounded convex hulls}\label{sec:convex_hull_unconstraint}
Let $q:\mathscr{D}\rightarrow\R$, $(y,t)\mapsto q(y,t)$ be a given pressure function. We recall that the set of nonlinear constraints $K_{(y,t)}$ has been defined in \eqref{eq:nonlinear_constraints} as the set of tuples $z=(\mu,w,m,h,\sigma,e)\in Z$ satisfying
\begin{align*}
    \mu\in\set{\rho_-,\rho_+},\quad m=\mu w,\quad h=(e+q(x,t))w,\quad \mu w\otimes w-\sigma=\frac{2}{n}e\id.
\end{align*}
We point to \eqref{eq:matrix} for the definition of the matrix $\tilde{M}(z)$, and define
the open set
\begin{align*}
U:=\left\{z\in Z:\ \mu\in(\rho_-,\rho_+),~\lamax(\tilde M (z))<\frac{2}{n} e \right\}.
\end{align*}
One of the main goals of this subsection is the following characterization of the ($\Lambda$-)convex hull of $K_{(y,t)}$. The abstract definition of the $\Lambda$-convex hull can be found in Section \ref{sec:segments}. 

\begin{proposition}\label{prop:hull}
There holds \begin{align*}
K_{(y,t)}^\Lambda=K_{(y,t)}^{co}=\overline U.\end{align*}
\end{proposition}
Note that the relaxation $\overline{U}$ in the here considered unbounded case is independent of the considered pressure function $q(y,t)$.

We establish Proposition \ref{prop:hull} by iteratively relaxing the conditions of $K_{(y,t)}$. We begin at fixed $\mu$ with the relaxation of the homogeneous Euler equations, then, still at fixed $\mu$, apply the relaxation of the condition involving $h$, and in a final step relax the conditions involving $\mu$ and $m$.

By means of Proposition \ref{prop:hull} we obtain the full relaxation of globally dissipative Euler flows, in the sense that any sort of weak limit takes values in the set computed in Proposition \ref{prop:hull}, see Proposition \ref{prop:weak_limits}. In order to get a converse convex integration statement by means of a Baire category argument in the Tartar framework, we introduce similarly to \cite{GK-EE} for the homogeneous Euler equations and \cite{Markfelder_EE} for the compressible Euler equations, an $L^\infty$-bound on $e$. That is we consider  
\begin{equation}\label{eq:nonlinear_constraints_with_gamma}
K_{\gamma,(y,t)}:=\set{z\in K_{(y,t)}:e\leq \gamma}
\end{equation}
for a fixed $\gamma>0$. It is easy to check that any $z\in K_{\gamma,(y,t)}$ is bounded by means of $\rho_-$, $\rho_+$, $\gamma$ and $q(y,t)$, cf. \cite[Lemma 3(iii)]{DeL-Sz-Adm} for the part addressing $\sigma$. Therefore, we are also interested in computing $K_{\gamma,(y,t)}^\Lambda,$ or at least a large enough subset thereof.

Before starting to compute the hulls, we begin with some basics about $\Lambda$, $K_{(y,t)}$ and $K^\Lambda_{(y,t)}$.

In the following we will fix $q=q(y,t)$ and drop the $(y,t)$ dependence in our notation.

\subsubsection{Step 0. - Preliminaries}\label{sec:segments}

For $K'\subset Z$ the associated $\Lambda$-convex hull $(K')^{\Lambda}$ is defined as the set of points $z\in Z$ such that for all $\Lambda$-convex functions $f:Z\rightarrow\R$ there holds $f(z)\leq \sup_{z'\in K'}f(z')$, see \cite{Kirchheim,Matousek_Plechac_1998}.

As indicated, the proof of Proposition \ref{prop:hull} is based on the computation of $\Lambda$-segments starting in $K$, which relies on the fact that if we define the set of first order $\Lambda$-segments 
\[
(K')^{\Lambda,1}:=K'\cup \set{sz_1+(1-s)z_2:z_1,z_2\in K',~s\in[0,1],~z_1-z_2\in\Lambda}
\]
of a set $K'\subset Z$, then for any $K''\subset (K')^\Lambda$ there holds $(K'')^{\Lambda,1}\subset (K')^\Lambda$.

In Lemma \ref{lem:wave_cone_spans_Z} below we will quickly verify that our wave cone $\Lambda$, recall \eqref{eq:wave_cone} for the definition, spans all of $Z$. As a consequence $(K')^\Lambda$ is a closed set for any $K'\subset Z$, cf. \cite[Corollary 2.4]{Matousek_Plechac_1998}. Thus we are allowed to take closures in the process of computing $K^{\Lambda}$.

\begin{lemma}\label{lem:wave_cone_spans_Z}
The linear span of $\Lambda$ is $Z$.
\end{lemma}
\begin{proof}
The following $\bar{z}=(\bar{\mu},\bar{w},\bar{m},\bar{h},\bar{\sigma},\bar{e})$ are contained in $\Lambda$:
\begin{gather*}
    (\bar{\mu},0,0,0,0,0),~\bar{\mu}\neq 0,\\
    (0,0,\bar{m},0,0,0),~
    (0,\bar{m},\bar{m},0,0,0),~
    (0,0,\bar{m},\bar{m},0,0),~\bar{m}\neq 0,\\
    (0,0,0,0,\bar{\sigma},\bar{e}),~\bar{\sigma}\neq 0,~-\frac{2}{n}\bar{e}\text{ eigenvalue of }\bar{\sigma}.
\end{gather*}
The first four vectors clearly generate any $z$ of the form $z=(\mu,w,m,h,0,0)$. Moreover, since any trace-free, symmetric $\bar{\sigma}\neq 0$ has at least one positive and one negative eigenvalue, the vectors stated in the last line generate any $z$ of the form $(0,0,0,0,\sigma,e)$.
\end{proof}

The next preliminary statement addresses the size of $\Lambda$ with respect to the constraints $K$. It will be used 
for the convex integration sketched in Section \ref{sec:conclusion}.

 \begin{lemma}\label{lem:bigg_cone}
 For any $z_1,z_2\in K$ with 
 $z_1\neq z_2$, we have $\bar{z}=z_2-z_1\in\Lambda$.
 \end{lemma}
 \begin{proof}
 \textbf{Case 1.} $\mu_1=\mu_2$.
 
 Since $z_i\in K$, without loss of generality we have $$z_i=(\rho_+,w_i,\rho_+w_i,(e_i+q)w_i,\rho_+ w_i\otimes w_i-(2/n) e_i\id ,e_i),$$ $e_i=\rho_+\frac{1}{2}\abs{w_i}^2$, $i=1,2$, and therefore $\bar{w}\neq 0$. In particular we get $\bar \mu=0$ and $\bar m =\rho_+\bar w$.
  
  Let $\xi\in \bar{w}^\perp$.
 If $\bar{e}=0$, then all that needs to be checked is that $\bar{\sigma}\xi=c\rho_+\bar{w}$, for some $c\in\R$. There holds
 \begin{align*}
 \bar{\sigma}\xi=\rho_+(w_2\otimes w_2-w_1\otimes w_1)\xi=\rho_+(w_2\otimes \bar{w}+\bar{w}\otimes w_1)\xi=\rho_+(w_1\cdot\xi)\bar{w},
 \end{align*}
 so $\bar{z}\in\Lambda$ follows.
 
  If $\bar{e}\neq 0$, we similarly obtain from  $z_1,z_2\in K$ that
  $$(\bar{\sigma}+(2/n)\bar{e}\id)\xi=\rho_+(w_2\otimes w_2-w_1\otimes w_1)\xi=\rho_+(w_1\cdot\xi)\bar{w},$$
  so it remains to check that $\bar{h}\cdot\xi=(w_1\cdot\xi)\bar{e}$ also holds. We have
  $$\bar{h}\cdot\xi=(e_2w_2-e_1w_1+q\bar{w})\cdot\xi=(e_2\bar{w}+\bar{e}w_1)\cdot\xi=(w_1\cdot\xi)\bar{e},$$
  the result then follows.
  
   \textbf{Case 2.} $\mu_1\neq \mu_2$. 
   
   Without loss of generality assume that $\mu_1=\rho_-$, $\mu_2=\rho_+$.
    Let $\xi\in \bar{w}^\perp$. Through a simple calculation one obtains that
    \begin{align*}
    \bar \mu (\bar \sigma + (2/n)\bar e \id)=\bar \mu (\mu_2 w_2\otimes w_2-\mu_1 w_1\otimes w_1)=\bar m \otimes \bar m -\rho_-\rho_+ \bar w\otimes \bar w.
    \end{align*}
    Hence, we have $ (\bar \sigma + (2/n)\bar e \id)\xi + c\bar m=0$ for $c=-\frac{1}{\bar \mu}\bar m\cdot \xi$. Therefore, $\bar m \cdot \xi + c \bar\mu=0$ is also satisfied automatically and there remains only to check $\bar h \cdot \xi + c \bar e=0$. As in the previous case, we have
    $$\bar{h}\cdot\xi=(e_2w_2-e_1w_1+q\bar{w})\cdot\xi=(e_2\bar{w}+\bar{e}w_1)\cdot\xi=(w_1\cdot\xi)\bar{e},$$
    and from 
    $\rho_+w_2-\rho_-w_1=\rho_+\bar w+\bar \mu w_1$
    it follows $w_1\cdot \xi=\frac{1}{\bar \mu}\bar m \cdot \xi$, which finishes the proof of the lemma.
 \end{proof}

Solely by this property we have the following result.
 \begin{corollary}\label{cor:seg}
 Let $K'\subset K$ be a compact set. For any $z\in\text{int} (K')^{co}$ there exists $\bar{z}\in \Lambda$ such that
 $$[z-\bar z,z+\bar z]\subset \text{int} (K')^{co}\text{ and }|\bar z|\geq\frac{1}{2N}d(z,K'),$$
 where $N=\text{dim}(Z)$ and $d$ is the Euclidean distance on $Z$.
 \end{corollary}
 The proof is the same as those of \cite[Lemma 6]{DeL-Sz-Adm}, respectively \cite[Lemma 4.9]{GKSz}, relying on Carath\'eodory's theorem and Lemma \ref{lem:bigg_cone} above, therefore we omit it.

\subsection{Computing the hulls}\label{sec:hulls}
Let $z\in Z$. For brevity we introduce also the notation
\begin{align}
\label{eq:M0}M_0(z)&:=\mu w\otimes w-\sigma-\frac{2}{n}e\id,\\
\label{eq:eta}\eta(z)&:=\frac{h-(e+q)w}{\abs{h-(e+q)w}},
\end{align}
the latter only in the case of $h\neq (e+q)w$.

\subsubsection{Step 1. - Homogeneous Euler relaxation}

In this step we implement a strategy similar to \cite{GKSz} for the mass and momentum parts of the relaxation. However, contrary to said paper, here we do not have boundedness, so we will work with laminates once more (instead of a generalized Krein-Milman argument).
More precisely, we have the following.

\begin{lemma}\label{lem:step1}
Let $z\in Z$ with $\mu\in\set{\rho_-,\rho_+}$, $m=\mu w$, $h=(e+q)w$ and $\lamax(M_0(z))\leq 0$. Then $z\in K^\Lambda$. If in addition $e\leq \gamma$, then $z\in K^\Lambda_\gamma$.
\end{lemma}
\begin{proof}
Let us define
\begin{align*}
\hat U:=\left\{z\in Z:\ \mu\in\{\rho_-,\rho_+\},\ m=\mu w,\ h=(e+q)w,\ \lamax(M_0(z)) \leq 0\right\},\end{align*}
and for $j=0,\ldots,n$ the sets
\begin{align*}
\tilde K_j:=\left\{z\in \hat U:\ M_0(z)\text{ has exactly }j\text{ eigenvalues strictly less than }0\right\}.
\end{align*}
Note that, by the definition of $\hat U$, the other eigenvalues must be equal to $0$. Furthermore, $\tilde K_0=K$ and $\hat U=\cup_{j=0}^n \tilde K_j.$

Let $j\geq 1$ and $z\in \tilde K_j$.
Consider the case $\mu=\rho_+$, and assume
\begin{align*}
\mu w\otimes w -\sigma = \sum_{i=1}^n \lambda_i \phi_i\otimes \phi_i,
\end{align*}
with
$\lambda_1\geq\ldots\geq \lambda_n$, and $\phi_1,\ldots,\phi_n$ an orthonormal basis of $\R^n$. Then by assumption we have $\lambda_i= \frac{2}{n}e$ for $i=1,\ldots,n-j$, and $\lambda_i<\frac{2}{n}e$ for $i=n-j+1,\ldots,n$.

Define $\bar z \in Z$ as follows. Let
\begin{align*}
\bar \mu = 0,\quad \bar w = \phi_n,\quad \bar m=\rho_+ \phi_n,\quad \bar h=(e+q)\phi_n,\quad \bar e=0, \\
\bar{\sigma}=\rho_+\phi_n\otimes w+\rho_+w\otimes \phi_n-2\rho_+(w\cdot \phi_n) \phi_n\otimes \phi_n,
\end{align*}
It is easy to check that $\bar z\in\ \Lambda$. Indeed, \eqref{eq:wave_cone} reduces to having a non-trivial $(\xi,c)\in\R^{n+1}$ such that
\begin{align*}
\begin{pmatrix}
\bar{\sigma} & \rho_+\phi_n\\
\phi_n & 0
\end{pmatrix}
\begin{pmatrix}
\xi \\ c
\end{pmatrix}=0.
\end{align*}
This however easily follows for our choice of $\bar\sigma$, by picking any $\xi$ such that $\xi\cdot \phi_n=0$, and then setting $c=-w\cdot \xi$.

It follows that
\begin{align*}
\rho_+(w+s\bar{w})\otimes &(w+s\bar{w})-(\sigma+s\bar{\sigma})\\
&=\sum_{j=1}^n\lambda_j\phi_j\otimes \phi_j+s(\rho_+\phi_n\otimes w+\rho_+w\otimes \phi_n-\bar{\sigma})+s^2\rho_+\phi_n\otimes \phi_n\\
&=\sum_{j=1}^{n-1}\lambda_j\phi_j\otimes \phi_j + (\lambda_n+2\rho_+(w\cdot \phi_n) s+\rho_+s^2)\phi_n\otimes \phi_n.
\end{align*}
Since $\lambda_n<\frac{2}{n}e$, there exist $s_1<0<s_2$ roots of $ \lambda_n+2\rho_+(w\cdot \phi_n) s+\rho_+s^2=\frac{2}{n}e$. 
If we denote $z_s=z+s\bar z$, the eigenvalues of $\rho_+ w_s\otimes w_s-\sigma_s$, $s=s_{1,2}$ are $\lambda_1,\ldots,\lambda_{n-1}$ and $\frac{2}{n}e$. Furthermore, $\mu_s=\rho_+$, $m_s=\mu_s w_s$ and $h_s-(e_s+q)w_s$ remain  unchanged for all $s\in[s_1,s_2]$.

Consequently, we have shown that $z$ lies on the segment $[z+s_1\bar z, z +s_2\bar z]$, whose endpoints now lie in $\tilde K_{j-1}$, from where it follows that
$$\tilde K_{j}\subset (\tilde K_{j-1})^{\Lambda,1}.$$
Since $\tilde{K}_0=K$, one may then iterate this argument to obtain that
$$\tilde K_{j}\subset K^{\Lambda},$$
for any $j\geq 1$. 
Using $\hat U=\cup_{j=0}^n \tilde K_j$, we conclude the statement of the lemma in the unbounded case. However, by noting that we only used directions satisfying $\bar e=0$,  the $\gamma$ bound is also immediate.
\end{proof}

\subsubsection{Step 2. - The dissipative part}
Next we relax the condition $h=(e+q)w$. Recall \eqref{eq:eta} for the definition of $\eta(z)$.
\begin{lemma}\label{lem:dis}
Let $z\in Z$ with $\mu\in\set{\rho_-,\rho_+}$, $m=\mu w$, $h\neq (e+q)w$ and $\lamax(M_0(z))<0$. Then $z\in K^\Lambda$. Furthermore, there exists $\alpha_\gamma(z)>0$ such that
if in addition $e<\gamma$ and
\begin{align*}
    \lamax\big(M_0(z)+\mu\alpha_\gamma(z)\abs{h-(e+q)w}\eta(z)\otimes\eta(z)\big)\leq 0,
\end{align*}
then $z\in K^\Lambda_\gamma$. In \eqref{eq:explicit_alpha} the quantity $\alpha_\gamma(z)$ is stated explicitly.
\end{lemma}
\begin{proof}
Let $\alpha>0$. We consider the direction
\begin{gather*}
    \bar{\mu}=0,\quad \bar{e}=1,\quad \bar{w}=\alpha\eta(z),\quad \bar{m}=\mu\bar{w},\\
    \bar{h}=w+\bar{w}(e+q+2\beta),\quad \beta=\frac{1-\mu w\cdot\bar{w}}{\mu\abs{\bar{w}}^2},\\
    \bar{\sigma}=\mu\big(w\otimes\bar{w}+\bar{w}\otimes w+2\beta \bar{w}\otimes\bar{w}\big)^{\circ}.
\end{gather*}
Then $\bar{z}\in \Lambda$, since now \eqref{eq:wave_cone} reduces to having a non-trivial $(\xi,c)\in\R^{n+1}$ such that
\begin{align*}
\begin{pmatrix}
\bar{\sigma}+\frac{2}{n}\id & \mu\bar w\\
\bar w & 0\\
w & 1
\end{pmatrix}
\begin{pmatrix}
\xi \\ c
\end{pmatrix}=0.
\end{align*}
This however follows once more by picking any $\xi\in \bar w^\perp$ and setting $c:=-w\cdot \xi$.

Clearly we have
\begin{align*}
    \mu+s\bar{\mu}\in\set{\rho_-,\rho_+},\quad m+s\bar{m}=(\mu+s\bar{\mu})(w+s\bar{w})
\end{align*}
on the whole segment $z+s\bar{z}$, $s\in\R$.

Looking at the change of $h-(e+q)w$ along the segment we have
\begin{align*}
    h+s\bar{h}&-(e+s\bar{e}+q)(w+s\bar{w})=h-(e+q)w+s(\bar{h}-(e+q)\bar{w}-w)-s^2\bar{w}\\&=\bar{w}\left(\frac{\abs{h-(e+q)w}}{\alpha}+2\beta s-s^2\right).
\end{align*}
Since $\alpha>0$, we find $s_1<0<s_2$ characterized by
\begin{align}\label{eq:step2_quadratic_equation_for_s}
    s^2-2\beta s-\frac{\abs{h-(e+q)w}}{\alpha}=0,
\end{align}
for which $z+s\bar{z}$ has no error in the $h$-component when compared to $(e+q)w$. 

In order to conclude that $z+s_{1,2}\bar{z}$ lie in $K^\Lambda$, and therefore $z\in K^\Lambda$, it remains to have $\lamax(M_0(z+s_{1,2}\bar{z}))\leq 0$, cf. Lemma \ref{lem:step1}. By the definition of $\bar{\sigma}$, $\beta$ and by \eqref{eq:step2_quadratic_equation_for_s}, we find for $s=s_{1,2}$ that
\begin{align*}
    M_0(z+s\bar{z})&=M_0(z)+s\left(\mu w\otimes\bar{w}+\mu\bar{w}\otimes w-\bar{\sigma}-\frac{2}{n}\id\right)+s^2\mu\bar{w}\otimes\bar{w}\\
    &=M_0(z)+(s^2-2\beta s)\mu\bar{w}\otimes\bar{w}\\
    &=M_0(z)+\frac{\abs{h-(e+q)w}}{\alpha}\mu\bar{w}\otimes\bar{w}.
\end{align*}
By assumption, the matrix $M_0(z)$ is strictly negative definite, we therefore can chose $\alpha$ sufficiently large such that also $\lamax(M_0(z+s\bar{z}))\leq 0$. This finishes the proof in the unbounded case.

In the bounded case we in addition need to satisfy $e+s\leq \gamma$ for $s=s_{1,2}$ in order to conclude that $z\in K^\Lambda_\gamma$. Since $s_1<0$, the latter condition is satisfied in particular when 
\[
s_2=\gamma-e.
\]
By \eqref{eq:step2_quadratic_equation_for_s}, this is the case if and only if
\begin{align}\label{eq:gamma_e_is_root}
    (\gamma-e)^2-2\beta(\gamma-e)-\frac{\abs{\Delta h}}{\alpha}=0,
\end{align}
where we abbreviate
\begin{align*}
    \Delta h:=h-(e+q)w.
\end{align*}
By definition
\begin{align*}
    \beta=\frac{1-\mu w\cdot\bar{w}}{\mu\abs{\bar{w}}^2}=\frac{1-\mu\alpha\eta(z)\cdot w}{\mu\alpha^2}.
\end{align*}
Thus, \eqref{eq:gamma_e_is_root} is equivalent to
\begin{align}\label{eq:alphaquad}
    \alpha^2+\frac{2w\cdot\eta(z)(\gamma-e)-\abs{\Delta h}}{(\gamma-e)^2}\alpha-\frac{2}{\mu(\gamma-e)}=0.
\end{align}
Since $\alpha$ has to be positive, this condition forces us to choose
\begin{align}\label{eq:explicit_alpha}
    \alpha=\frac{\abs{\Delta h}-2w\cdot\eta(z)(\gamma-e)}{2(\gamma-e)^2}+\left(\frac{\big(\abs{\Delta h}-2w\cdot\eta(z)(\gamma-e)\big)^2}{4(\gamma-e)^4}+\frac{2}{\mu(\gamma-e)}\right)^{\frac{1}{2}}.
\end{align}
Let us define this term as $\alpha_\gamma(z)$. Going back to the expression for $M_0(z+s\bar{z})$, we therefore find
\begin{align}\label{eq:bestcond}
    \lamax\big(M_0(z)+\mu\alpha_\gamma(z)\abs{\Delta h}\eta(z)\otimes\eta(z)\big)\leq 0
\end{align}
as a sufficient condition to conclude $z\in K_\gamma^\Lambda$.
\end{proof}

\begin{remark}\label{rem:EE} Setting $\mu=\rho_-=\rho_+=1$ Lemma \ref{lem:dis} in particular improves the results of \cite{GK-EE} for the globally dissipative homogeneous Euler equations with respect to two aspects: it generalizes the proofs of that paper to dimensions higher than two, and it states a more explicit and thus larger subset of the hull $K^\Lambda_\gamma$.
\end{remark}

As a consequence of Lemma \ref{lem:dis} and the fact that $K^\Lambda$ is a closed set, we immediately deduce the following statement for the unbounded case.
\begin{corollary}\label{cor:unbounded_case1}
Any $z\in Z$ with $\mu\in\set{\rho_-,\rho_+}$, $m=\mu w$ and $\lamax(M_0(z))\leq 0$ belongs to $K^\Lambda$.
\end{corollary}
In the bounded case we instead conclude the following sufficient condition for being in the hull $K^\Lambda_\gamma$.
\begin{corollary}\label{cor:condition_for_being_in_K_gamma_hull}
Assume that $z\in Z$ satisfies $\mu\in\{\rho_-,\rho_+\}$, $m=\mu w$ and
\begin{align}\label{eq:suff_condition_in_proposition}
\lamax(\mu w\otimes w-\sigma)+\varepsilon\abs{h-(e+q)w}<\frac{2}{n} e
\end{align}
for some $\varepsilon>0$. Then $z\in K^\Lambda_\gamma$ for any $\gamma\geq \gamma_\varepsilon(e)$ where $\gamma_\varepsilon:\R_+\rightarrow\R_+$ is the continuous function
\begin{align}\label{eq:gammaeps}
    \gamma_\varepsilon(e)=e+18\varepsilon^{-2}\max\set{\varepsilon\sqrt{\rho_+e},\rho_+}.
\end{align}
\end{corollary}
\begin{proof}
Let $z\in Z$ be as stated. If $h=(e+q)w$, the statement follows for any choice of $\gamma\geq e$ in view of Lemma \ref{lem:step1}. Assume $h\neq (e+q)w$. We will show that for $\gamma$ as stated in \eqref{eq:gammaeps} the quantity $\alpha_\gamma(z)$ from Lemma \ref{lem:dis} satisfies $\mu \alpha_\gamma(z)\leq \varepsilon$. As a consequence of said lemma, $z\in K^\Lambda_\gamma$.

For the definition of $\gamma$ note that \eqref{eq:suff_condition_in_proposition} implies 
\begin{align*}
    \abs{h-(e+q)w}<\frac{2e}{n\varepsilon},
\end{align*}
as well as
\begin{align*}
    \sqrt{\mu}\abs{w}&=\big(\tr(\mu w\otimes w-\sigma)\big)^{\frac{1}{2}}< \sqrt{2e}.
\end{align*}
We then estimate the explicit expression \eqref{eq:explicit_alpha} for $\alpha_\gamma(z)$ as follows:
\begin{align*}
    \mu\alpha_\gamma(z)&\leq \mu\frac{\abs{h-(e+q)w}}{(\gamma-e)^2}+2\sqrt{\mu}\frac{\sqrt{\mu}\abs{w}}{\gamma-e}+\frac{\sqrt{2\mu}}{\sqrt{\gamma-e}}\\
    &\leq \frac{2\rho_+e}{n\varepsilon(\gamma-e)^2}+\frac{\sqrt{8\rho_+e}}{\gamma-e}+\frac{\sqrt{2\rho_+}}{\sqrt{\gamma-e}}.
\end{align*}
Now it is easy to check that each of the three terms on the right-hand side is $\leq \varepsilon /3$ when picking $\gamma\geq \gamma_\varepsilon(e)$.
\end{proof}

\subsubsection{Step 3. - The density}

Let us now add a Muskat-type direction, in order to obtain $K^\Lambda$, as well as a characterization of a large part of $K_\gamma^\Lambda$. 
In order to formulate the statement we define for $z\in Z_0$, i.e. $z\in Z$ with $\mu\in(\rho_-,\rho_+)$, the expression
\begin{align*}
    \tilde e(z):=\frac{T_+(z)-T_-(z)}{\rho_+-\rho_-},
\end{align*}
where
$$T_\pm (z):=\frac{\rho_\pm}{2}\frac{|m-\rho_\mp w|^2}{(\mu-\rho_\mp)^2}$$ are the quantities involved in the trace inequalities from \cite{GKSz}.

\begin{proposition}\label{prop:condition_for_being_in_K_gamma_hull2}
Let $z\in Z_0$. If 
$\lamax(\tilde M (z))\leq \frac{2}{n}e$, then $z\in K^\Lambda$. If in addition
\begin{align}\label{eq:suff_condition_in_proposition2}
\lamax(\tilde M (z))+\varepsilon\abs{h-(e+q)w-\tilde e(z)(m-\mu w)}<\frac{2}{n} e
\end{align}
is satisfied for some $\varepsilon>0$,
then for
\begin{align}\label{eq:gamma_prop_last}
\gamma=\max\{\gamma_\varepsilon(e+(\rho_+-\mu)\tilde e(z)),\gamma_\varepsilon(e+(\rho_--\mu)\tilde e(z))\}+1,
\end{align}
where $\gamma_\varepsilon$ is the map given by Corollary \ref{cor:condition_for_being_in_K_gamma_hull}, there holds
$z\in \text{int} (K_\gamma^\Lambda)$.
\end{proposition}
\begin{proof}
Let $z\in Z_0$ with $\lamax(\tilde{M}(z))\leq \frac{2}{n}e$.
We first will show that such a $z$ lies on a $\Lambda$-segment with endpoints satisfying the conditions of Corollary \ref{cor:unbounded_case1}, which implies $z\in K^\Lambda$. To do so, we consider a Muskat-direction $\tilde{z}(z)\in Z$ defined via
\begin{gather*}
\tilde{\mu}=1,\quad
\tilde{w}(z)=\frac{m-\mu w}{(\rho_+-\mu)(\mu-\rho_-)},\quad
\tilde{m}(z)=w+(\rho_++\rho_--\mu)\tilde{w}(z),\\ 
\tilde{\sigma}(z)+\frac{2}{n}\tilde{e}(z)\id =\tilde{m}(z)\otimes \tilde{m}(z)-\rho_+\rho_-\tilde{w}(z)\otimes\tilde{w}(z),\\ 
\tilde{h}(z)=\tilde{e}(z)w+\big(e+q+(\rho_++\rho_--2\mu)\tilde{e}(z)\big)\tilde{w}(z).
\end{gather*}
Here the definition of $\tilde{e}$ and $\tilde{\sigma}$ is understood as decomposition into trace and traceless part. It is easy to check that $\tilde z (z)\in \Lambda$. Furthermore, a simple calculation yields that in fact the $\tilde e(z)$ defined this way is exactly $\frac{T_+(z)-T_-(z)}{\rho_+-\rho_-}$.

Let $s\in(\rho_--\mu,\rho_+-\mu)$ and $z_s:=z+s\tilde{z}(z)$. First of all observe that
\begin{align*}
(\rho_+-\mu-s)(\mu+s-\rho_-)\tilde{w}(z_s)&=m+s\tilde{m}(z)-(\mu+s)(w+s\tilde{w}(z))\\
&=m-\mu w +s(\rho_++\rho_--2\mu)\tilde{w}(z)-s^2\tilde{w}(z)\\
&=(\rho_+-\mu-s)(\mu+s-\rho_-)\tilde{w}(z).
\end{align*}
Hence $\tilde{w}(z_s)=\tilde{w}(z)$ for all $s\in\R$ (strictly speaking for $s\neq \rho_\pm-\mu$, but in these points we take the canonical extension), and
\begin{align*}
\tilde{m}(z_s)&=w+s\tilde{w}(z)+(\rho_++\rho_--\mu-s)\tilde{w}(z_s)\\
&=w+(\rho_++\rho_--\mu)\tilde{w}(z)=\tilde{m}(z).
\end{align*}
Note that the invariances $\tilde{\sigma}(z_s)=\tilde{\sigma}(z)$ and $\tilde{e}(z_s)=\tilde{e}(z)$ follow by the definition of $\tilde{\sigma}$, $\tilde{e}$.

Plugging 
\begin{align*}
w&=\tilde{m}(z)-(\rho_++\rho_--\mu)\tilde{w}(z),\\
m&=\mu w +(\rho_+-\mu)(\mu-\rho_-)\tilde{w}(z)=\mu \tilde{m}(z)-\rho_-\rho_+\tilde{w}(z)
\end{align*}
into the definition of $\tilde M(z)$, cf. \eqref{eq:matrix}, leads us to
\begin{align}\label{eq:tM}
\begin{split}
\tilde M(z)&=\mu\tilde{m}(z)\otimes \tilde{m}(z)-\rho_-\rho_+(\tilde{m}(z)\otimes\tilde{w}(z)+\tilde{w}(z)\otimes \tilde{m}(z))\\
&\phantom{=asd}+\rho_-\rho_+(\rho_++\rho_--\mu)\tilde{w}(z)\otimes\tilde{w}(z)-\sigma.
\end{split}
\end{align}
Thus,
\begin{multline*}
\tilde M(z_s)-\frac{2}{n}e_s \id \\= \tilde M(z)-\frac{2}{n}e\id + s\left(\tilde{m}(z)\otimes\tilde{m}(z)-\rho_-\rho_+\tilde{w}(z)\otimes \tilde{w}(z)-\tilde{\sigma}(z)-\frac{2}{n}\tilde{e}(z)\id\right)
\\=\tilde M(z)-\frac{2}{n}e\id
\end{multline*}
by the definition of $\tilde{\sigma},\tilde{e}$. In particular $\lamax(\tilde{M}(z_s))\leq \frac{2}{n}e_s$ by assumption.

Furthermore, we also have
\begin{align*}
m_s-\mu_s w_s&=m-\mu w + s(w+(\rho_++\rho_--\mu)\tilde w(z) - w - \mu\tilde w(z))-s^2 \tilde w (z)\\&=\left(1+\frac{\rho_++\rho_--2\mu}{(\rho_+-\mu)(\mu-\rho_-)}s -\frac{1}{(\rho_+-\mu)(\mu-\rho_-)}s^2\right)(m-\mu w).
\end{align*}
Therefore, at the endpoints we have for $z_\pm=z+(\rho_\pm-\mu)\tilde z (z)$ the following:
\begin{align*}
&\mu_\pm=\mu+(\rho_\pm-\mu)\tilde \mu = \rho_\pm,\\
&m_\pm-\mu_\pm w_\pm=0,
\end{align*}
and finally, using \eqref{eq:tM} as well as $\tilde m(z_\pm)=w_\pm+\rho_{\mp}\tilde w(z_\pm)$, we get that
\begin{align*}
\tilde M (z_\pm)=\rho_\pm(w_\pm\otimes w_\pm+\rho_\mp(w_\pm\otimes \tilde w(z_\pm)+\tilde w(z_\pm)\otimes w_\pm)+\rho_\mp^2\tilde w(z_\pm)\otimes \tilde w(z_\pm))\\-\rho_-\rho_+((w_\pm+\rho_\mp\tilde w(z_\pm))\otimes \tilde w(z_\pm)+\tilde w(z_\pm)\otimes (w_\pm+\rho_\mp\tilde w(z_\pm)))\\+\rho_-\rho_+ \rho_\mp\tilde w(z_\pm)\otimes \tilde w(z_\pm)-\sigma_\pm=\rho_\pm w_\pm\otimes w_\pm-\sigma_\pm.
\end{align*}
Thus, we may conclude $z_\pm\in K^\Lambda$ via Corollary \ref{cor:unbounded_case1}, and hence also $z\in K^\Lambda$.

Let us now turn to the bounded case. Due to the choice of $\tilde h$, we have in addition that 
\begin{align*}
h_\pm-(e_\pm+q)w_\pm=h&-(e+q)w+(\rho_\pm-\mu)(\rho_++\rho_--2\mu) \tilde e(z) \tilde{w}(z)\\
&-(\rho_\pm-\mu)^2\tilde e(z) \tilde{w}(z)=h-(e+q)w-\tilde e(z)(m-\mu w).
\end{align*}
Thus we have shown that, due to \eqref{eq:suff_condition_in_proposition2}, $z_\pm$ satisfy \eqref{eq:suff_condition_in_proposition}, and hence all the conditions of Corollary \ref{cor:condition_for_being_in_K_gamma_hull}. Noting that $e_\pm=e+(\rho_\pm-\mu)\tilde e(z)$, and thus choosing $\tilde{\gamma}=\tilde{\gamma}_\varepsilon(z)$ as the maximum of the two values $\gamma_\varepsilon(e+(\rho_\pm-\mu)\tilde e(z))$ yields $z\in K_{\tilde{\gamma}}^\Lambda$.

In order to see that the point $z$ in fact belongs to the interior $\text{int}(K^\Lambda)$, one observes that \eqref{eq:suff_condition_in_proposition2} and $\mu\in(\rho_-,\rho_+)$ remain valid in a small neighborhood of $z$. By constructing for each point $z'$ in such a neighborhood the corresponding Muskat segment, we conclude that  $z'\in K^{\Lambda}_{\tilde{\gamma}_\varepsilon(z)}$. By continuity of $\tilde{\gamma}_\varepsilon$ we can pick a uniform $\gamma$ in a sufficiently small neighborhood of $z$.
\end{proof}

\begin{proof}[Proof of Proposition \ref{prop:hull}]
The characterization of the full ($\Lambda$-)convex hull in the unbounded case follows from the first half of Proposition \ref{prop:condition_for_being_in_K_gamma_hull2} by taking the closure and the convexity of $z\mapsto\lamax(\tilde M (z))$ shown in \cite[Lemma 4.3]{GKSz}.
\end{proof}

In the bounded case we finish the investigation of the hull by observing that the expression for $\gamma$ in \eqref{eq:gamma_prop_last} can be slightly simplified.
Let us denote the two quantities appearing in $\gamma_\varepsilon$ in \eqref{eq:gamma_prop_last} by $Q_\pm(z)$, i.e. 
\[
Q_\pm(z):=e+(\rho_\pm-\mu)\tilde{e}(z).
\]
\begin{lemma}\label{lem:size_of_gamma_prop_last}
If $\mu\in(\rho_-,\rho_+)$ and \eqref{eq:suff_condition_in_proposition2} hold true, then
\begin{align*}
    \max\set{e,T_+(z),T_-(z)}\leq \max\set{Q_+(z),Q_-(z)}\leq 2\max\set{e,T_+(z),T_-(z)}.
\end{align*}
\end{lemma}
\begin{proof}
The upper inequality follows in a straightforward way by dropping the negative term in each of the expressions. It does not rely on \eqref{eq:suff_condition_in_proposition2}. For the lower inequality one first of all observes that
\begin{align*}
    e=\frac{\mu-\rho_-}{\rho_+-\rho_-}Q_+(z)+\frac{\rho_+-\mu}{\rho_+-\rho_-}Q_-(z)\leq \max\set{Q_+(z),Q_-(z)}.
\end{align*}
Regarding $T_\pm(z)$ one computes
\begin{align}\label{eq:convex_combination_of_Tpm}
    \frac{\mu-\rho_-}{\rho_+-\rho_-}T_+(z)+\frac{\rho_+-\mu}{\rho_+-\rho_-}T_-(z)=\frac{1}{2}\tr(\tilde{M}(z))\leq e,
\end{align}
where the  latter inequality is a consequence of \eqref{eq:suff_condition_in_proposition2}. It follows that
\begin{align*}
    Q_+(z)=e+\frac{\rho_+-\mu}{\rho_+-\rho_-}(T_+(z)-T_-(z))\geq T_+(z)
\end{align*}
and similarly $Q_-(z)\geq T_-(z)$.
\end{proof}

\begin{remark}\label{rem:simpler_gamma}
In view of Lemma \ref{lem:size_of_gamma_prop_last} and the monotonicity of $\gamma_\varepsilon$, there is no loss in replacing the $\gamma$ stated in \eqref{eq:gamma_prop_last} by
\begin{align}\label{eq:simpler_gamma}
    \gamma=\gamma_{\varepsilon}\left(2\max\set{e,T_+(z),T_-(z)}\right)+1.
\end{align}
\end{remark}

\subsection{Conclusion}\label{sec:conclusion}

We have the following more abstract version of Theorem \ref{thm:main_theorem}.
\begin{theorem}\label{thm:abstract_CI_theorem}
Let $\gamma>0$. Theorem \ref{thm:main_theorem} remains true when condition \eqref{eq:sufficient_condition_for_Linfty_hull} is replaced by
\begin{align}\label{eq:abstract_condition_for_hull}
    (\mu,w,m,h,\sigma,e)(y,t)\in \text{int} \left( K^{co}_{\gamma,(y,t)}\right)\text{ for all }(y,t)\in\mathscr{U}.
\end{align}
\end{theorem}
Here we mean that Theorem \ref{thm:main_theorem} holds true word by word except that \eqref{eq:sufficient_condition_for_Linfty_hull} is swapped with \eqref{eq:abstract_condition_for_hull}. For further clarification we point out that the set $K_{\gamma,(x,t)}$ is defined via \eqref{eq:nonlinear_constraints}, \eqref{eq:nonlinear_constraints_with_gamma} with respect to the pressure $q$ induced by $(\mu,w,m,h,\sigma,e)$, cf. Definition \ref{def:weaksolslin}.
\begin{proof}[Proof of Theorem \ref{thm:abstract_CI_theorem}]
We start by observing that if $q$ was constant in $\mathscr{U}$, then \eqref{eq:linear_system_with_p} together with the set of constraints $K_\gamma$ would  fit into the framework stated in the appendix of \cite{Sz_ipm}. In fact, due to Lemma \ref{lem:locpw} and Corollary \ref{cor:seg} above, the result would directly follow  from \cite[Theorem 5.1]{Sz_ipm}.

However, the Tartar framework can easily be adapted to the case when the set of constraints also depends on $(x,t)$. The extra condition which is needed is that the map
$(x,t)\mapsto K_{\gamma,(x,t)}$ is continuous and bounded on $\mathscr{U}$ with respect to the Hausdorff metric $d_{\mathcal{H}}$, see \cite{Crippa_Gusev_Spirito_Wiedemann}. This continuity however, is a consequence of the fact that $q$ is assumed to be continuous and bounded on $\mathscr{U}$, see for instance \cite{Crippa_Gusev_Spirito_Wiedemann,GK-EE,GKSz}.
\end{proof}

\begin{proof}[Proof of Theorem \ref{thm:main_theorem}]
Let $(z,q)$, $z=(\mu,w,m,h,\sigma,e)$ be as stated in Theorem \ref{thm:main_theorem}. In view of Theorem \ref{thm:abstract_CI_theorem} it only remains to verify that there exists a value $\gamma>0$ for which \eqref{eq:abstract_condition_for_hull} holds true.

By Proposition \ref{prop:condition_for_being_in_K_gamma_hull2},  Remark \ref{rem:simpler_gamma}, and since $e,T_+,T_-\in L^\infty(\mathscr{U})$ by assumption, we can pick 
\begin{align*}
    \gamma:=\gamma_\varepsilon\left(2\max\set{\norm{e}_{L^\infty(\mathscr{U})},\norm{T_+}_{L^\infty(\mathscr{U})},\norm{T_-}_{L^\infty(\mathscr{U})}}\right)+1,
\end{align*}
where the function $\gamma_\varepsilon$ is given by Corollary \ref{cor:condition_for_being_in_K_gamma_hull} and $\varepsilon>0$ is taken from \eqref{eq:sufficient_condition_for_Linfty_hull}.
\end{proof}

\section{Application to Rayleigh-Taylor}\label{sec:subsolutions}
We now turn to the construction of subsolutions for the horizontally flat initial data \eqref{eq:flat_initial_data}. The construction is mostly done for the system on the corresponding accelerated domain $\mathscr{D}$, cf. \eqref{eq:definition_accelerated_domain}.

We will begin by summarizing consequences of the one-dimensionality assumption of Definition \ref{def:1Dsubs}. Thereafter, we impose relation \eqref{eq:general_m_with_F} with some general uniformly convex $F$ between $m_n$ and $\mu$, take $\mu$ as the unique entropy solution of the resulting hyperbolic conservation law, and compute the kinetic and potential energy (im)balance in dependence of the chosen constitutive relation $F$. Finally, we show that a) chosing $F_\lambda$, $\lambda\in(0,1/2)$ as stated in Proposition \ref{prop:subsgen} indeed leeds to non-strict  subsolution that in the case of the periodic channel can be perturbed to a strict one, and b) that $F_{1/3}$ maximizes the total energy dissipation among all considered convex relations between $\mu$ and $m_n$.

\subsection{Consequences of one-dimensionality}
Let $\Omega=(0,1)^{n-1}\times(-L,L)$ or $\Omega=\T^{n-1}\times(-L,L)$, $n=2,3$. 
In accordance with Definition \ref{def:1Dsubs} we seek to construct a (non-strict) subsolution $z$ with pressure $q$ and dissipation measure $\nu$ satisfying \eqref{eq:1Dsub1}--\eqref{eq:choice2}. Throughout the construction we will always assume that the mixing zone does not reach the top or bottom boundaries. This means that near those boundaries the subsolution will always be a solution with $w=gte_n$. Consequently, there are no issues with the boundary conditions \eqref{eq:boundary_condition_w}, \eqref{eq:new_boundary_conditions}.

The one-dimensionality assumption in particular means that  $\sigma$ is specified once $\mu$ and $m_n$ are constructed. Indeed, depending on the dimension we have
\begin{align}\label{eq:choice22}
    \sigma=\frac{1}{2}\begin{pmatrix}
    -f & 0\\
    0 &f
    \end{pmatrix},\quad\sigma=\frac{1}{3}\begin{pmatrix}
    -f & 0 &0\\
    0&-f&0\\
    0&0&2f
    \end{pmatrix},
\end{align}
with 
\begin{align}\begin{split}\label{eq:f_rewritten}
    f&=\frac{\mu\rho_-\rho_+g^2t^2-2\rho_-\rho_+ m_n gt+(\rho_++\rho_--\mu)m_n^2}{(\rho_+-\mu)(\mu-\rho_-)}\\
    &=\frac{(\rho_++\rho_--\mu)(m_n-\mu gt)^2}{(\rho_+-\mu)(\mu-\rho_-)}+2gtm_n-\mu g^2t^2.\end{split}
\end{align}

Next, one observes that in the one-dimensional setting the first $n-1$ components of the linear variable momentum balance in \eqref{eq:linear_system_with_p} trivially hold true with this choice of $\sigma$. The remaining component reduces to
\begin{gather}
    \partial_t m_n+\partial_{y_n}\left(\sigma_{nn}+\frac{2}{n}e+q\right)=0.\label{eq:linear_momentum_balance_2}
\end{gather}
Once $m_n$, $\mu$ (and thus $\sigma_{nn}$) and $e$ have been specified, the latter equation can always be solved by a suitable pressure function $q$. 

Similarly, one can take the linear variable energy balance
\begin{align}\label{eq:linear_energy_balance}
    \partial_te+\partial_{y_n}h_n=\nu
\end{align}
as the definition for $\nu$, once $e$ and $h_n$ are specified. In the end of course we have to verify that $\nu\leq 0$ in accordance with \eqref{eq:euler_local_energy_inequality}, \eqref{eq:euler_0g_local_energy_inequality}.

Concerning the equations, it thus for now remains to look at 
\begin{align}\label{eq:remaining_equation_mu}
    \partial_t\mu+\partial_{y_n}m_n&=0.
\end{align}

\subsection{\texorpdfstring{A constitutive law for $m_n$}{A constitutive law}}
Next we turn to the choice of $m_n$. Note that outside the mixing zone we need to have $m_n=\mu gt$. To specify the constitutive law for $m_n$ we therefore make the ansatz
\begin{align}\label{eq:constitutive_law_m}
    m_n=gt\big(\mu+F(\mu)\big)
\end{align}
for some uniformly convex $F:[\rho_-,\rho_+]\rightarrow \R$ of class $\cC^2$ and vanishing at the endpoints $\rho_\pm$.
By \eqref{eq:remaining_equation_mu} this means that $\mu$ is characterized through the hyperbolic conservation law
\begin{align*}
    \partial_t\mu+gt\partial_{y_n}\big(\mu+F(\mu)\big)=0,
\end{align*}
or going back to original coordinates $(x_n,t)=(y_n-gt^2/2,t)$ and $\rho(x_n,t)=\mu(y_n,t)$, that
\begin{align}\label{eq:conservation_law_rho}
    \partial_t\rho+gt\partial_{x_n}\big(F(\rho))=0.
\end{align}

The convexity assumption on $F$ implies that the Riemann problem \eqref{eq:conservation_law_rho} with $\rho(0,\cdot)=\rho_0$ has a unique entropy solution given by 
\begin{align}\label{eq:entropy_solution_rho}
    \rho(x_n,t)=\begin{cases}
    \rho_+,&x_n>F'(\rho_+)\frac{gt^2}{2},\\
    (F')^{-1}\left(\frac{2x_n}{gt^2}\right),&F'(\rho_-)\frac{gt^2}{2}<x_n<F'(\rho_+)\frac{gt^2}{2},\\
    \rho_-,&x_n<F'(\rho_-)\frac{gt^2}{2}.
    \end{cases}
\end{align}
By our assumptions there holds 
\begin{align}\label{eq:assumptions_F}
F'(\rho_-)< 0< F'(\rho_+).
\end{align}

With the help of this entropy solution $\rho$ we define $\mu=\mu_F$ by
\begin{align*}
    \mu(y_n,t)=\rho(y_n-gt^2/2,t).
\end{align*}
As mentioned earlier, the time span $(0,T)$ for the induced subsolution is taken such that the mixing zone does not touch the boundary of the domain, i.e. $T=T(F,L)$ such that
\begin{align*}
    \max\set{-F'(\rho_-),F'(\rho_+)}\frac{gT^2}{2}<L.
\end{align*}

Note here that taking the entropy solution is  itself a second choice, besides imposing the relation \eqref{eq:constitutive_law_m}, due to the fact that other, less regular solutions to the Riemann problem are also conceivable. 

\subsection{Kinetic and potential energy}

In order to make a further selection in the constitutive law \eqref{eq:constitutive_law_m}, we look at the total kinetic and potential energy induced by $F(\mu)$. 
The subsolution energies are defined as
\begin{align*}
    E_{kin}(t)&:=\int_{(-L,L)+\frac{1}{2}gt^2}e(y_n,t)-m_n(y_n,t)gt+\frac{1}{2}\mu(y_n,t) g^2t^2\:dy_n,\\
    E_{pot}(t)&:=\int_{(-L,L)+\frac{1}{2}gt^2}\mu(y_n,t)g\left(y_n-\frac{1}{2}gt^2\right)\:dy_n,\\
    E_{tot}(t)&:=E_{kin}(t)+E_{pot}(t).
\end{align*}
The definition of the potential energy agrees with the usual potential energy in original variables after undoing the transformation \eqref{eq:transf}. The definition of $E_{kin}$ is justified in view of Lemma \ref{lem:total_energy} and the definition of $E_{pot}$, which together precisely say that $E_{tot}(t)$ agrees with the total energy of the later induced solutions $(\rho_{sol},v_{sol})$.

We will now rewrite $E_{kin}(t)$ and $E_{pot}(t)$ in terms of the function $F$. Regarding the kinetic energy we observe that \eqref{eq:sufficient_condition_for_Linfty_hull},  \eqref{eq:form_of_tilde_M_in_subs} and \eqref{eq:choice2} imply
\begin{align*}
    e\geq \frac{n}{2}\lamax\left(\tilde{M}(z)\right)=\frac{1}{2}f(t,\mu,m_n).
\end{align*}
For convex integration the gap between $e$ and $\frac{1}{2}f$ has to be strict inside the mixing zone, but can be taken arbitrarily small.
It follows that
\begin{align*}
    E_{kin}(t)&\geq \int_{(-L,L)+\frac{1}{2}gt^2}\frac{1}{2}f(t,\mu,m_n)-m_ngt+\frac{1}{2}\mu g^2t^2\:dy_n=:E_{kin}^*(t).
\end{align*}
Note that $E_{kin}(t)=E_{kin}^*(t)$ for non-strict subsolutions.

Plugging our ansatz \eqref{eq:constitutive_law_m} into \eqref{eq:f_rewritten}, one obtains
\begin{align*}
    E_{kin}^*(t)&=\frac{g^2t^2}{2}\int_{(-L,L)+\frac{1}{2}gt^2}\frac{\rho_++\rho_--\mu}{(\rho_+-\mu)(\mu-\rho_-)}F(\mu)^2\:dy_n=\frac{g^2t^2}{2}\int_{-L}^La(\rho)^{-1}F(\rho)^2\:dx_n.
\end{align*}
Here we switched to original coordinates and abbreviated
\begin{align}\label{eq:definition_of_a}
    a(r):=\frac{(\rho_+-r)(r-\rho_-)}{\rho_++\rho_--r}.
\end{align}

Since $F$ is Lipschitz continuous and vanishing at the endpoints $\rho_\pm$, the integrant in $E_{kin}^*$ is vanishing outside the mixing zone associated with $\rho$ given through \eqref{eq:entropy_solution_rho}. Inside the mixing zone we use
\[
x_n=\frac{gt^2}{2}F'(\rho(x_n,t))
\]
as a change of coordinates to find 
\begin{align}\label{eq:expression_E_kin_star}
    E_{kin}^*(t)&=\frac{g^3t^4}{4}\int_{\rho_-}^{\rho_+}a(r)^{-1}F(r)^2F''(r)\:dr=-\frac{g^3t^4}{4}\int_{\rho_-}^{\rho_+}\frac{d}{dr}\left(a(r)^{-1}F(r)^2\right)F'(r)\:dr.
\end{align}

In a similar way we rewrite the difference of total potential energy. This time we can directly start in $(x,t)$ variables to find 
\begin{align*}
    E_{pot}(t)-E_{pot}(0)&=\int_{-L}^L(\rho-\rho_0)gx_n\:dx_n\\
    &=\frac{g^3t^4}{4}\int_{\rho_-}^{\rho_+}\big(r-\rho_0(F'(r)gt^2/2)\big)F'(r)F''(r)\:dr.
\end{align*}
By \eqref{eq:assumptions_F} and the monotonicity of $F$ it follows that \begin{align}\label{eq:expression_E_pot}
    E_{pot}(t)-E_{pot}(0)&=-\frac{g^3t^4}{4}\int_{\rho_-}^{\rho_+}\frac{1}{2}F'(r)^2\:dr.
\end{align}

Now we are able to estimate the total energy balance from below
\begin{align}
\begin{split}\label{eq:energy_balance_with_F}
    E_{tot}(t)-&E_{tot}(0)\geq E^*_{kin}(t)+E_{pot}(t)-E_{pot}(0)\\
    &=-\frac{g^3t^4}{4}\int_{\rho_-}^{\rho_+}\frac{d}{dr}\left(a(r)^{-1}F(r)^2\right)F'(r)+\frac{1}{2}F'(r)^2\:dr.
\end{split}
\end{align}
Again, there holds equality for non-strict subsolutions. Note that this shows the first part of Proposition \ref{prop:energies_of_subsolution_sec2}. We turn to the second part, i.e. the maximization of the functional on the right-hand side with respect to $F$, in Section \ref{sec:variational_problem}. 

At this point one observes even without a detailed investigation of this variational problem that there is a choice that stands out, namely
\begin{align*}
    F_\lambda(r):=-\lambda a(r)=-\lambda \frac{(\rho_+-r)(r-\rho_-)}{\rho_++\rho_--r},~\lambda>0,
\end{align*}
which can be checked to be indeed a convex function vanishing at $\rho_\pm$. This choice
turns \eqref{eq:energy_balance_with_F} into
\begin{align}\begin{split}\label{eq:total_energy_balance}
    E^*_{kin}(t)+E_{pot}(t)-E_{pot}(0)&=\frac{g^3t^4}{4}\left(\lambda-\frac{1}{2}\right)\lambda^2\int_{\rho_-}^{\rho_+}a'(r)^2\:dr\\
    &=\frac{g^3t^4(\rho_+-\rho_-)^3}{12\rho_+\rho_-}\left(\lambda-\frac{1}{2}\right)\lambda^2.
\end{split}
\end{align}
We also note that
\begin{align}\label{eq:computation_E_pot}
    E_{pot}(t)-E_{pot}(0)=-\frac{\lambda^2g^3t^4}{24}\frac{(\rho_+-\rho_-)^3}{\rho_+\rho_-}.
\end{align}

In the following subsection we will study the corresponding local energy balance and thus show that we indeed obtain a non-strict subsolution induced by $F_\lambda$, $\lambda\leq \frac{1}{2}$.

\subsection{\texorpdfstring{Non-strict subsolution for $F_\lambda$}{Local dissipation}}

Fix $\lambda\in(0,\frac{1}{2}]$ and consider $\mu$, $m$ induced by $F_\lambda$ as described.
In accordance with \eqref{eq:nonstrictsubs} for non-strict one-dimensional subsolutions, we further set
\begin{align}\label{eq:edef_hdef}
    e:=\frac{1}{2}f,\quad h:=\left((e+q)gt+\frac{T_+-T_-}{\rho_+-\rho_-}(m_n-\mu gt)\right)e_n,
\end{align}
where $T_\pm=T_\pm(z)$ have been defined in \eqref{eq:traces}. Note that now all components of the subsolution, as well as the pressure $q$ and $\nu$ are determined, cf. \eqref{eq:linear_momentum_balance_2}, \eqref{eq:linear_energy_balance}. Our goal is to compute the latter.

\begin{lemma}\label{lem:computing_nu}
There holds
\begin{align*}
    \nu=(1-2\lambda)g(m_n-\mu gt)=(1-2\lambda)g^2tF_\lambda(\mu)\leq 0.
\end{align*}
\end{lemma}
\begin{proof}
Let us go back to the $(t,y_n)$ coordinates to obtain from \eqref{eq:f_rewritten} that, with the choice of $F_\lambda$ and $m_n=gt(\mu+F_\lambda(\mu))$, we have
\begin{align}\label{eq:f_in_computation_lemma}
    f=gt m_n + (1-\lambda)gt(m_n-\mu gt),
\end{align}
as well as
\begin{align}\label{eq:Tpm_in_computation_lemma}
    T_\pm=g^2 t^2 \frac{\rho_\pm}{2}\left(1-\lambda\frac{\rho_\pm-\mu}{\rho_++\rho_--\mu}\right)^2,
\end{align}
and hence
\begin{align*}
    \frac{T_+-T_-}{\rho_+-\rho_-}=\frac{g^2 t^2}{2}\left((1-\lambda)^2-\lambda^2 \frac{\rho_+\rho_-}{(\rho_++\rho_--\mu)^2} \right).
\end{align*}
Consequently
\begin{align}\label{eq:hdef}
    h= \left((e+q)gt +\frac{g^2 t^2}{2}\left((1-\lambda)^2-\lambda^2 \frac{\rho_+\rho_-}{(\rho_++\rho_--\mu)^2} \right)(m_n-\mu gt) \right)e_n.
\end{align}
Now observe that \eqref{eq:choice22}, \eqref{eq:linear_momentum_balance_2} and \eqref{eq:edef_hdef} yield
\begin{align}\begin{split}\label{eq:derivative_e_and_q}
    \partial_{y_n}(e+q)&=\partial_{y_n}\left(\frac{2}{n}e+q\right)+\frac{n-2}{n}\partial_{y_n}e=-\partial_tm_n-\partial_{y_n}\sigma_{nn}+\frac{n-2}{n}\partial_{y_n}e\\
    &=-\partial_tm_n-\frac{1}{2}\partial_{y_n}f.
    \end{split}
\end{align}
This allows to calculate
\begin{align*}
    \nu=\partial_t e+\partial_{y_n} h_n =\frac{1}{2}\partial_t f&-gt \partial_t m_n-\frac{1}{2}gt\partial_{y_n}f\\
    +&\frac{g^2 t^2}{2}\left((1-\lambda)^2-\lambda^2 \frac{\rho_+\rho_-}{(\rho_++\rho_--\mu)^2} \right)\partial_{y_n}(m_n-\mu gt)\\
    -&\frac{g^2 t^2}{2}\lambda^2\rho_+\rho_-\partial_{y_n}\left( \frac{1}{(\rho_++\rho_--\mu)^2} \right)(m_n-\mu gt).
\end{align*}
A further simple calculation based on \eqref{eq:f_in_computation_lemma} and $\partial_t\mu+\partial_{y_n}m_n=0$ yields
\begin{align*}
    \frac{1}{2}\partial_t f-gt \partial_t m_n-\frac{1}{2}gt\partial_{y_n}f=\frac{1}{2}g(2-\lambda)(m_n-&\mu gt)-\frac{\lambda}{2}gt\partial_t(m_n-\mu gt)\\&-\frac{g^2t^2}{2}(1-\lambda)\partial_{y_n}(m_n-\mu gt),
\end{align*}
from where using that 
$$
m_n-\mu gt=F_\lambda(\mu) gt,\quad\partial_t(m_n-\mu gt)=\frac{m_n-\mu gt}{t}+gt\partial_t(F_\lambda(\mu)),
$$ plugging back the expression for $F_\lambda$ and rearranging appropriate terms, we get that
\begin{align}\label{eq:nucalc}
\begin{split}
    \nu&=-\lambda(1-\lambda)g^2 t \frac{(\rho_+-\mu)(\mu-\rho_-)}{\rho_++\rho_--\mu}+\lambda^2 \frac{g^2 t^2}{2}\partial_t\left(\frac{(\rho_+-\mu)(\mu-\rho_-)}{\rho_++\rho_--\mu} \right)\\
    &\phantom{=}+\lambda^2 \frac{g^3 t^3}{2} \left(1-\lambda+\lambda \frac{\rho_+\rho_-}{(\rho_++\rho_--\mu)^2} \right)\partial_{y_n}\left(\frac{(\rho_+-\mu)(\mu-\rho_-)}{\rho_++\rho_--\mu} \right)\\
    &\phantom{=}+\lambda^3 \frac{g^3 t^3}{2} \rho_+\rho_- \frac{(\rho_+-\mu)(\mu-\rho_-)}{\rho_++\rho_--\mu}\partial_{y_n}\left( \frac{1}{(\rho_++\rho_--\mu)^2} \right).
    \end{split}
\end{align}
Note that this implies $\nu=0$ outside the mixing zone.

Let us now use the fact that $\mu$ is an entropy solution of
$$\partial_t\mu+gt\partial_{y_n}\big(\tilde F_\lambda(\mu)\big)=0,$$
with $\tilde F_\lambda(r)=r+F_\lambda(r)$. Hence we have
$$\mu(t,y_n)=(\tilde F'_\lambda)^{-1}\left(\frac{2y_n}{gt^2}\right),$$
inside the mixing zone. 

We further have
\begin{align*}
    \tilde F'_\lambda(\mu)=1-\lambda + \lambda \frac{ \rho_+\rho_-}{(\rho_++\rho_--\mu)^2},\quad \tilde F''_\lambda(\mu)=F''_\lambda(\mu)=2\lambda \frac{ \rho_+\rho_-}{(\rho_++\rho_--\mu)^3}.
\end{align*}
Using the rule for differentiating an inverse function, one then gets
\begin{align*}
    &\partial_t \mu = \frac{1}{F''_\lambda(\mu)}(-2)\frac{2y_n}{gt^3}=-\frac{2y_n (\rho_++\rho_--\mu)^3}{\lambda  \rho_+\rho_- gt^3},\\
    &\partial_{y_n} \mu=\frac{1}{F''_\lambda(\mu)}\frac{2}{gt^2}=\frac{ (\rho_++\rho_--\mu)^3}{\lambda  \rho_+\rho_- gt^2}.
\end{align*}
This allows us to compute
\begin{align*}
 &\partial_{y_n}\left( \frac{1}{(\rho_++\rho_--\mu)^2} \right)=\frac{ 2}{\lambda  \rho_+\rho_- gt^2},\\
    &\partial_{y_n}\left(\frac{(\rho_+-\mu)(\mu-\rho_-)}{\rho_++\rho_--\mu} \right)=\left(1-\frac{ \rho_+\rho_-}{(\rho_++\rho_--\mu)^2} \right)\frac{ (\rho_++\rho_--\mu)^3}{\lambda  \rho_+\rho_- gt^2},\\
    &\partial_{t}\left(\frac{(\rho_+-\mu)(\mu-\rho_-)}{\rho_++\rho_--\mu} \right)=-\left(1-\frac{ \rho_+\rho_-}{(\rho_++\rho_--\mu)^2} \right)\frac{2y_n (\rho_++\rho_--\mu)^3}{\lambda  \rho_+\rho_- gt^3}.
\end{align*}
Plugging back into \eqref{eq:nucalc} we get
\begin{align*}
    \nu=&-\lambda(1-\lambda)g^2 t \frac{(\rho_+-\mu)(\mu-\rho_-)}{\rho_++\rho_--\mu}\\&-\lambda^2 \frac{g^2 t^2}{2}\left(1-\frac{ \rho_+\rho_-}{(\rho_++\rho_--\mu)^2} \right)\frac{2y_n (\rho_++\rho_--\mu)^3}{\lambda  \rho_+\rho_- gt^3}\\
    &+\lambda^2 \frac{g^3 t^3}{2} \left(1-\lambda+\lambda \frac{\rho_+\rho_-}{(\rho_++\rho_--\mu)^2} \right)\left(1-\frac{ \rho_+\rho_-}{(\rho_++\rho_--\mu)^2} \right)\frac{ (\rho_++\rho_--\mu)^3}{\lambda  \rho_+\rho_- gt^2}\\
    &+\lambda^3 \frac{g^3 t^3}{2} \rho_+\rho_- \frac{(\rho_+-\mu)(\mu-\rho_-)}{\rho_++\rho_--\mu}\frac{ 2}{\lambda  \rho_+\rho_- gt^2}.
\end{align*}
One can observe that the middle two lines above fall out. Indeed, from $\mu(t,y_n)=(\tilde F'_\lambda)^{-1}\left(\frac{2y_n}{gt^2}\right)$ one has
\begin{align}\label{eq:expression_y}
-\frac{2y_n}{gt^2}+1-\lambda+\lambda \frac{ \rho_+\rho_-}{(\rho_++\rho_--\mu)^2}=0.
\end{align}
Hence we are left with
\begin{align*}\nu&=-\lambda(1-\lambda)g^2 t \frac{(\rho_+-\mu)(\mu-\rho_-)}{\rho_++\rho_--\mu}+\lambda^3 \frac{g^3 t^3}{2} \rho_+\rho_- \frac{(\rho_+-\mu)(\mu-\rho_-)}{\rho_++\rho_--\mu}\frac{ 2}{\lambda  \rho_+\rho_- gt^2}\\
&=\lambda(2\lambda-1)g^2 t \frac{(\rho_+-\mu)(\mu-\rho_-)}{\rho_++\rho_--\mu},
\end{align*}
which is indeed non-positive for $\lambda\in[0,\frac{1}{2}]$, and in agreement with the stated expression.
\end{proof}

\begin{proof}[Proof of Proposition \ref{prop:subsgen}.] Since $\mu$ is given as a rarefaction wave, it is continuous except at the initial time at the initial interface. Since $m$ is defined with an additional factor $t$, it is continuous everywhere. Next it is easy to see that $f$, cf. \eqref{eq:f_in_computation_lemma}, is continuous across the boundary of the mixing zone with $f=\mu g^2t^2$ outside. Thus $e,\sigma\in\cC^0(\overline{\mathscr{D}})$.
After integrating \eqref{eq:linear_momentum_balance_2} in $y_n$ and using 
\begin{align*}
    \partial_tm_n-g(\mu+F_\lambda(\mu))=g^2t^2(1+F_\lambda'(\mu))^2\partial_{y_n}\mu,
\end{align*}
one can also see that $q$ and thus $h$ are continuous on all of $\overline{\mathscr{D}}$. The continuity of $\nu$ follows from Lemma \ref{lem:computing_nu}.

Moreover, the properties of a one-dimensional non-strict subsolution follow by construction. For \eqref{eq:Tpm_and_q_bounded}, see \eqref{eq:Tpm_in_computation_lemma}. Finally \eqref{eq:total_enerrgy_in_subs_prop} has been computed in \eqref{eq:total_energy_balance}.
\end{proof}

\subsection{Passing to a strict subsolution}

Let us now consider only the periodic case $\Omega=\T^{n-1}\times(-L,L)$ and $\lambda\in(0,1/2)$ fixed.
In order to have a strict subsolution inside the mixing zone, we may modify the above construction simply by redefining
\begin{align*}
    e:=\frac{1}{2}f+\delta t^2(\rho_+-\mu)^2(\mu-\rho_-)^2,
\end{align*}
for $\delta>0$ small enough. For the rest of the construction we leave $\mu,m_n,\sigma$ unchanged. As such also $T_\pm$, cf. \eqref{eq:Tpm_in_computation_lemma}, remains unchanged. We define once more $q$ through \eqref{eq:linear_momentum_balance_2} and $h$  via \eqref{eq:hdef},  both with respect to the modified $e$.
With these choices \eqref{eq:sufficient_condition_for_Linfty_hull} holds true, for instance with $\varepsilon=1$.

Observe that although the expression for $h$ depends on $e$, $\partial_{y_n} h_n$ does not, since $\partial_{y_n}(e+q)$ remains unchanged in view of \eqref{eq:derivative_e_and_q}.
We then have, using the calculation of the previous subsection, that
\begin{align*}
    \nu:=&\partial_t e+\partial_{y_n} h_n =\lambda(2\lambda-1)g^2 t \frac{(\rho_+-\mu)(\mu-\rho_-)}{\rho_++\rho_--\mu}+2\delta t(\rho_+-\mu)^2(\mu-\rho_-)^2 \\
    &+2\delta t ^2(\rho_+-\mu)(\mu-\rho_-)(\rho_++\rho_--2\mu)\left(-\frac{2y_n(\rho_++\rho_--\mu)^3}{\lambda\rho_+\rho_- gt^3}\right)
    \\=& \left[ 2\delta\left( (\rho_+-\mu)(\mu-\rho_-)-(\rho_++\rho_--2\mu)\frac{2y_n(\rho_++\rho_--\mu)^3}{\lambda\rho_+\rho_- gt^2} \right) \right.\\&+ \left. \frac{\lambda(2\lambda-1)g^2}{\rho_++\rho_--\mu}\right]t (\rho_+-\mu)(\mu-\rho_-).
\end{align*}
By \eqref{eq:expression_y} we may write
\begin{align*}
    \frac{2y_n(\rho_++\rho_--\mu)^3}{\lambda\rho_+\rho_- gt^2}=& \frac{(\rho_++\rho_--\mu)^3}{\lambda\rho_+\rho_- }\left(1-\lambda+\lambda \frac{ \rho_+\rho_-}{(\rho_++\rho_--\mu)^2} \right)\\=&(1-\lambda)\frac{(\rho_++\rho_--\mu)^3}{\lambda\rho_+\rho_- }+\rho_++\rho_--\mu,
\end{align*}
which is bounded for $\mu\in[\rho_-,\rho_+]$, and consequently we get
$$ 0<C:=\max_{\mu\in[\rho_-,\rho_+]}\left|(\rho_+-\mu)(\mu-\rho_-)-(\rho_++\rho_--2\mu)\frac{2y_n(\rho_++\rho_--\mu)^3}{\lambda\rho_+\rho_- gt^2}\right|<+\infty.$$
We may pick $0<\delta<-\frac{1}{2C}\frac{\lambda(2\lambda-1)g^2}{\rho_-}$ to obtain that
\begin{align*}
    \nu \leq t (\rho_+-\mu)(\mu-\rho_-) \left( 2C + \frac{\lambda(2\lambda-1)g^2}{\rho_++\rho_--\mu}\right)\leq 0.
\end{align*}
This concludes the construction of a strict subsolution which globally dissipates energy, associated with any $F_\lambda$ for $\lambda\in(0,1/2)$. Finally, concerning the error with respect to the old dissipation measure we may observe that
\begin{align*}
    |\nu-(1-2\lambda)g(m_n-\mu gt)|\leq 2C\delta t(\rho_+-\mu)(\mu-\rho_-).
\end{align*}
This concludes the proof of Lemma \ref{lem:strict_subs}.

\subsection{\texorpdfstring{Variational problem for $F$}{Variational problem for F}}\label{sec:variational_problem}

This subsection further emphasizes the special role of the family $F_\lambda$, or rather the particular member $F_{\frac{1}{3}}$, among all non-strict one-dimensional subsolutions satisfying a constitutive relation 
\[
m_n=gt(\mu +F(\mu))
\]
for some $F\in\cC^2([\rho_-,\rho_+])$ uniformly convex and vanishing at the endpoints. We already have seen in \eqref{eq:energy_balance_with_F} that the associated total energy balance reads
\begin{align*}
    E^*_{kin}(t)+E_{pot}(t)-E_{pot}(0)=-\frac{g^3t^4}{4}\cI(F),
\end{align*}
where
\begin{gather*}
\cI(F):=\int_{\rho_-}^{\rho_+}\frac{d}{dr}\left(a(r)^{-1}F(r)^2\right)F'(r)+\frac{1}{2}F'(r)^2\:dr,\\
a(r)=\frac{(\rho_+-r)(r-\rho_-)}{\rho_++\rho_--r}.
\end{gather*}
Imposing at this point maximal total energy dissipation as a selection criterion amounts to maximizing the functional $\cI$. The main goal of this subsection is to show that this variational problem has $F_{\frac{1}{3}}$ as its unique global maximizer. This then completes the proof of Proposition \ref{prop:energies_of_subsolution_sec2}.

We abbreviate the set of functions containing all convex (not necessarily uniform and not necessarily $\cC^2$) $F:[\rho_-,\rho_+]\rightarrow\R$ with $F(\rho_-)=F(\rho_+)=0$ by $\cF$.
\begin{proposition}\label{prop:variational_problem}
There holds
\[
\sup_{F\in\cF}\cI(F)=\cI(-a/3)=\frac{1}{54}\int_{\rho_-}^{\rho_+}a'(r)^2\:dr=\frac{1}{54}\frac{(\rho_+-\rho_-)^3}{3\rho_+\rho_-}
\]
and $F=-a/3$ is the only function in $\cF$ for which the supremum is attained.
\end{proposition}

We begin with some basic observations. First of all we note that $a$ is a smooth concave function on $[\rho_-,\rho_+]$ with simple zeros at the end points.

Let $F\in\cF$. Since $F:[\rho_-,\rho_+]\rightarrow\R$ is Lipschitz continuous and vanishing at the endpoints, there exists a function $H\in L^\infty(\rho_-,\rho_+)\cap W^{1,\infty}_{loc}(\rho_-,\rho_+)$, such that 
\begin{align*}
    F(r)=-a(r)H(r).
\end{align*}
Moreover, $F\leq 0$ by convexity and thus $H\geq 0$.

In terms of $H$ the functional $\cI$ can be rewritten as
\begin{align*}
   \cI(-aH)&= \int_{\rho_-}^{\rho_+}\frac{1}{2}(a'H+aH')^2-(aH^2)'(a'H+aH')\:dr\\
   &=\int_{\rho_-}^{\rho_+}\frac{1}{2}a^2(1-4H)(H')^2+\frac{1}{2}(a')^2H^2+aa'HH'-(a')^2H^3-3aa'H^2H'\:dr\\
   &=\int_{\rho_-}^{\rho_+}\frac{1}{2}a^2(1-4H)(H')^2-\frac{1}{2}aa''H^2(1-2H)\:dr\\
   &=:\cJ(H).
\end{align*}
We remark that although the derivative of $H$ might diverge towards the endpoints, the above integration by parts is justified due to the fact that $a'$, $H$ and hence also $aH'=-F'-a'H$ are essentially bounded on the full interval $(\rho_-,\rho_+)$.

For later use we will also introduce corresponding restricted functionals
\begin{align*}
    \cJ_{(r_1,r_2)}(H)&:=\int_{r_1}^{r_2}\frac{1}{2}a^2(1-4H)(H')^2-\frac{1}{2}aa''H^2(1-2H)\:dr,\\
    \cI_{(r_1,r_2)}(F)&:=\int_{r_1}^{r_2}(a(r)^{-1}F(r)^2)'F'(r)+\frac{1}{2}F'(r)^2\:dr,
\end{align*}
where $\rho_-\leq r_1<r_2\leq \rho_+$. Repeating the translation from $\cI$ to $\cJ$ but taking this time boundary terms into account, one finds
\begin{align}\label{eq:relation_between_restricted_functionals}
    \cI_{(r_1,r_2)}(-aH)=\cJ_{(r_1,r_2)}(H)+\left[aa'(H^2/2-H^3)\right]_{r=r_1}^{r_2}.
\end{align}

Defining the set 
\begin{align*}
    \cH&:=\set{H\in L^\infty(\rho_-,\rho_+)\cap W^{1,\infty}_{loc}(\rho_-,\rho_+):0\leq H,~-aH\text{ convex}},
\end{align*}
our discussion so far has shown the following statement.
\begin{lemma}\label{lem:suprema} There holds
\begin{align*}
    \sup_{F\in\cF}\cI(F)=\sup_{H\in\cH}\cJ(H).
\end{align*}
\end{lemma}
Thus Proposition \ref{prop:variational_problem} follows from the following lemma.
\begin{lemma}
For all $H\in \cH$ there holds $\cJ(H)\leq \cJ(1/3)$ and the inequality is strict for $H\neq 1/3$.
\end{lemma}
\begin{proof}
Let $H\in\cH$ and set $F=-aH$. For the comparison to $H=1/3$ we split
\begin{align}\begin{split}\label{eq:splitting_of_J}
    \cJ(H)-\cJ(&1/3)=\int_{\set{r:H(r)>1/4}}\frac{1}{2}a^2(1-4H)(H')^2-\frac{1}{2}aa''H^2(1-2H)+\frac{1}{54}aa''\:dr\\
    &+\int_{\set{r:H(r)<1/4}}\frac{1}{2}a^2(1-4H)(H')^2-\frac{1}{2}aa''H^2(1-2H)+\frac{1}{54}aa''\:dr.
    \end{split}
\end{align}
The first term can be estimated from above by $0$ as can be seen by $1-4H\leq 0$, $-aa''\geq 0$ and 
\begin{align*}
    \sup\set{x^2(1-2x):x\geq 0}=\left(\frac{1}{3}\right)^2\left(1-\frac{2}{3}\right)=\frac{1}{27}.
\end{align*}

For the second part, we write the domain of integration $\set{r:H(r)<1/4}$ as a disjoint union of maximal intervals such that at each boundary point $r_0$ of these intervals there holds $H(r_0)=\frac{1}{4}$ or $r_0\in\set{\rho_-,\rho_+}$. In any of the cases we can conclude
\begin{align}\label{eq:property_boundary_points}
    a(r_0)H(r_0)^k=a(r_0)\left(\frac{1}{4}\right)^k,~k\in \N.
\end{align}

Let now $(r_1,r_2)\subset\set{r:H(r)<\frac{1}{4}}$ be such a maximal interval. By \eqref{eq:relation_between_restricted_functionals} and \eqref{eq:property_boundary_points} we can rewrite
\begin{align}\label{eq:estimate_II}
    \cJ_{(r_1,r_2)}(H)-\cJ_{(r_1,r_2)}(1/3)&=\cI_{(r_1,r_2)}(F)-\cI_{(r_1,r_2)}(-a/3)+\left(\frac{1}{54}-\frac{1}{64}\right)\left[aa'\right]_{r=r_1}^{r_2}.
\end{align}

Now $F=-aH$ is a convex function with $-\frac{1}{4}a(r)\leq F(r)\leq 0$ for all $r\in (r_1,r_2)$. We claim that 
\begin{align}\label{eq:estimate_I}
    \cI_{(r_1,r_2)}(F)\leq \cI_{(r_1,r_2)}(-a/4)=\frac{1}{64}\int_{r_1}^{r_2}(a')^2\:dr.
\end{align}
Indeed, let us for simplicity assume that $F_{|(r_1,r_2)}$ is $\cC^2$. The general case can be justified by approximation, e.g. $(r_1,r_2)$ by $(r_1+\varepsilon,r_2-\varepsilon)$, $F$ and $a$ by mollified versions $F_\varepsilon$, $a_\varepsilon$, or directly by  interpreting $F''$ as a measure and the corresponding boundary terms in a trace sense.
Using \eqref{eq:property_boundary_points} we compute
\begin{align*}
    \cI_{(r_1,r_2)}(F)&=\int_{r_1}^{r_2}\left(\frac{F^2}{a}+\frac{F}{2}\right)'F'\:dr\\
    &=\int_{r_1}^{r_2}\left(\frac{F^2}{a}+\frac{F}{2}\right)(-F'')\:dr+\left[\left(-\frac{1}{16}a\right)F'\right]_{r=r_1}^{r_2},
\end{align*}
and observe that for each $\alpha>0$ the map $B_\alpha:\R\rightarrow\R$,
\[
B_\alpha(x)=\frac{x^2}{\alpha}+\frac{x}{2}
\]
is monotone increasing precisely on $[-\alpha/4,\infty)$. Since by assumption $F(r)\geq -a(r)/4$ for $r\in (r_1,r_2)$, and $-F''\leq 0$, there holds
\begin{align*}
    \cI_{(r_1,r_2)}(F)&\leq \int_{r_1}^{r_2}\left(-\frac{1}{16}a\right)(-F'')\:dr+\left[-\frac{1}{16}aF'\right]_{r=r_1}^{r_2}=-\int_{r_1}^{r_2}\frac{1}{16}a'F'\:dr\\
    &=\frac{1}{16}\int_{r_1}^{r_2}a''F\:dr+\left[-\frac{1}{16}a'F\right]_{r=r_1}^{r_2}\leq -\frac{1}{64}\int_{r_1}^{r_2}a''a\:dr+\left[\frac{1}{64}a'a\right]_{r=r_1}^{r_2}\\
    &=\frac{1}{64}\int_{r_1}^{r_2}(a')^2\:dr.
\end{align*}
In the second to last inequality we have used that $a$ is concave, and once more \eqref{eq:property_boundary_points}. This shows \eqref{eq:estimate_I}.

Using this estimate in \eqref{eq:estimate_II}, we find
\begin{align*}
    \cJ_{(r_1,r_2)}(H)-\cJ_{(r_1,r_2)}(1/3)&\leq \cI_{(r_1,r_2)}(-a/4)-\cI_{(r_1,r_2)}(-a/3)+\left(\frac{1}{54}-\frac{1}{64}\right)\left[aa'\right]_{r=r_1}^{r_2}\\
    &=\left(\frac{1}{64}-\frac{1}{54}\right)\left(\int_{r_1}^{r_2}(a')^2\:dr-\left[aa'\right]_{r=r_1}^{r_2}\right)\\
    &=\left(\frac{1}{64}-\frac{1}{54}\right)\int_{r_1}^{r_2}-aa''\:dr<0.
\end{align*}

This finishes the proof of $\cJ(H)\leq \cJ(1/3)$. The above estimate also shows that $\cJ(H)<\cJ(1/3)$ provided the set $\set{r:H(r)<1/4}$ is not empty. If it is empty, only the first integral contributes in \eqref{eq:splitting_of_J} and it is easy to see that one obtains a strict inequality except for $H=1/3$. 
\end{proof}

Our last remark addresses the selection criterion from \cite{Gimperlein_etal_1,Gimperlein_etal_2} mentioned in Section \ref{sec:selection}.
\begin{remark}\label{rem:laap}
If we form instead of the total energy $E^*_{kin}(t)+E_{pot}(t)$ the time derivative of the associated action, i.e.
\begin{align*}
    \frac{d}{dt}\cA(t):=E^*_{kin}(t)-E_{pot}(t),
\end{align*}
then it is easy to check that
\begin{align*}
    \frac{d}{dt}&\cA(t)+E_{pot}(0)=\\
    &\frac{g^3t^4}{4}\int_{\rho_-}^{\rho_+}\frac{1}{2}a^2(1+4H)(H')^2-\frac{1}{2}aa''H^2(1+2H)\:dr\geq 0.
\end{align*}
Thus $H=0$, i.e. $F=0$, is the strict minimizer of the above quantity. This means that the stationary interface $\rho(x,t)=\rho_0(x)$ is preferred by the strict least action admissibility criterion of \cite{Gimperlein_etal_2} when postulated in the set of subsolutions considered.
\end{remark}

\appendix

\section{Definitions and accelerated coordinates}\label{sec:defs}

For completeness we collect here rigorous definitions of (sub)solutions to the Euler equations with energy inequality, and a few words on the change of coordinates. Recall that $\Omega$ can be a bounded domain in $\R^n$, $n=2,3$, or $\Omega=\T^{n-1}\times(-L,L)$, and $T>0$. Since it is enough for our purposes to have essentially bounded velocity fields $v$, we refrain to formulate the definitions below in the slightly more general setting $v\in L^3(\Omega\times (0,T))$.

We also note that we further require the pressure to have an $L^1$ trace in the spatial variable, due to considerations of how the total energy balance behaves under the change of coordinates \eqref{eq:transf}, see also the proof of Lemma \ref{lem:translation}. As such, we also consider test functions in the distributional formulation of the equations which need not to vanish on the spatial boundary of the domain. We would like to point out that this somewhat unusual condition does not pose an issue for us as the convex integration solutions obtained by Theorem \ref{thm:main_theorem} inherit the pressure of the more regular subsolution. Also, the statement of Theorem \ref{thm:main_theorem} remains true with a more common weak formulation, but Lemma \ref{lem:total_energy} below might not.

\begin{definition}[Weak solutions]\label{def:weaksols}
Let $v_0\in L^\infty(\Omega;\R^n)$ be a weakly divergence-free vector field satisfying \eqref{eq:boundary_condition_v} in the sense of distributions, $\rho_0\in  L^\infty(\Omega)$ with $\rho_0\in\{\rho_-,\rho_+\}$ almost everywhere.
We say that the tuple $(\rho,v,p) \in L^\infty(\Omega\times(0,T);\R\times\R^n\times \R)$, $p\in L^1(\partial\Omega\times(0,T))$ is a weak solution to \eqref{eq:euler_only}, \eqref{eq:euler_local_energy_inequality}, \eqref{eq:boundary_condition_v} with initial data $(\rho_0,v_0)$, if for any test functions $\Phi\in \cC^\infty_c(\overline{\Omega}\times[0,T);\mathbb{R}^n)$,  $\Psi_1,\Psi_2\in \cC^\infty_c(\overline{\Omega}\times[0,T))$, $\Psi_2\geq 0$ we have
\begin{align*}
\int_0^T\int_{\Omega} \left[ (\rho v) \cdot \partial_t \Phi + (\rho v\otimes v+p\id ) :\nabla\Phi-g\rho\Phi_n \right] \ dx \ dt \hspace{25pt}&\\
-\int_0^T\int_{\partial\Omega}p\Phi\cdot\vec{n}\ dx\ dt
+\int_{\Omega}\rho_0(x)v_0 (x)\cdot \Phi(x,0)\ dx&=0,\\
\int_0^T\int_{\Omega} v\cdot\nabla \Psi_1\ dx\ dt&=0,\\
\int_0^T\int_{\Omega} \left[\rho \partial_t \Psi_1 + (\rho v)\cdot\nabla\Psi_1 \right] \ dx \ dt +\int_{\Omega} \rho_0 (x) \Psi_1(x,0)\ dx&=0,\\
\int_0^T\int_{\Omega} \left[\left(\frac{\rho|v|^2}{2} +\rho g x_n\right) \partial_t \Psi_2 + \left(\left(\frac{\rho|v|^2}{2}+\rho g x_n+p\right)v\right)\cdot\nabla\Psi_2\right] \ dx \ dt\hspace{25pt} &\\+\int_{\Omega} \left(\frac{\rho_0(x)|v_0(x)|^2}{2} +\rho_0(x) g x_n\right)\Psi_2(x,0)\ dx&\geq 0.
\end{align*}
The measure defined through the left-hand side of the last inequality is called the energy dissipation measure and denoted by $\nu$.
\end{definition}

Weak solutions to the transformed equations \eqref{eq:main0g}, \eqref{eq:boundary_condition_w}, \eqref{eq:euler_0g_local_energy_inequality}  are understood in the following similar sense. The accelerated domain $\mathscr{D}$ has been defined in \eqref{eq:definition_accelerated_domain}. We also set
\begin{align*}
    \overline{\mathscr{D}}^*:=\bigcup_{t\in[0,T)}\overline{\mathscr{D}(t)}\times\{t\},\quad \partial_{rel}\mathscr{D}:=\bigcup_{t\in(0,T)}\partial{\mathscr{D}(t)}\times\{t\}.
\end{align*}
\begin{definition}\label{def:weaksolstrans}
Let $\rho_0,v_0$ be as in Definition \ref{def:weaksols}.
The tuple $(\mu,w,q) \in L^\infty(\mathscr D;\R\times\R^n\times \R)$, $q\in L^1(\partial_{rel}\mathscr{D})$ is a weak solution to \eqref{eq:main0g}, \eqref{eq:boundary_condition_w}, \eqref{eq:euler_0g_local_energy_inequality} with initial data $(\rho_0,v_0)$, 
if
for any test functions $\Phi\in \cC^\infty_c(\overline{\mathscr{D}}^*;\mathbb{R}^n) $, $\Psi_1,\Psi_2\in\cC^\infty(\overline{\mathscr{D}}^*)$, $\Psi_2\geq 0$ we have
\begin{align*}
\int_{\mathscr D} \left[ (\mu w) \cdot \partial_t \Phi + (\mu w\otimes w+q\id) :\nabla\Phi \right] \ dx \ dt \hspace{89pt}& \\-\int_0^T\int_{\partial\mathscr{D}(t)}q\Phi\cdot\vec{n}\ dy\ dt +\int_{\Omega}\rho_0(x)v_0 (x)\cdot \Phi(x,0)\ dx&=0,\\
\int_{\mathscr D} \left[w-gte_n\right]\cdot\nabla \Psi_1\:dx\:dt&=0,\\
\int_{\mathscr D} \left[\mu \partial_t \Psi_1 + (\mu w)\cdot\nabla\Psi_1 \right] \ dx \ dt +\int_{\Omega} \rho_0 (x) \Psi_1(x,0)\ dx&=0,\\
\int_{\mathscr D} \left[\frac{\mu|w|^2}{2} \partial_t \Psi_2 + \left(\left(\frac{\mu|w|^2}{2}+q\right)w\right)\cdot\nabla\Psi_2\right] \ dx \ dt \hspace{65pt}&\\-\int_0^T gt\int_{\partial\mathscr{D}(t)}q {\Psi}_2e_n\cdot \vec{n}\ dx\ dt+\int_{\Omega} \frac{\rho_0(x)|v_0(x)|^2}{2} \Psi_2(x,0)\ dx&\geq 0.
\end{align*}
Here the local dissipation measure is once more defined in the same way as in Definition \ref{def:weaksols}.
\end{definition}

\begin{definition}\label{def:weaksolslin}
Let $\rho_0,v_0$ be as in Definition \ref{def:weaksols}, set $e_0:=\frac{1}{2}\rho_0\abs{v_0}^2$ and let $\nu\in\left(\cC_c^\infty(\overline{\mathscr{D}}^*)\right)^*$ be a negative measure. 
Similarly to above we say that the tuple $(\mu,w,m,h,\sigma,e) \in L^\infty(\mathscr D;\R\times\R^n\times\R^n\times\R^n\times\cS_0^{n\times n}\times\R)$ together with a pressure $q\in L^\infty(\mathscr D)\cap L^1(\partial_{rel}\mathscr{D})$ is a weak solution to \eqref{eq:boundary_condition_w}, \eqref{eq:linear_system_with_p} with initial data $(\rho_0,v_0,e_0)$, if for any test functions $\Phi\in \cC^\infty_c(\overline{\mathscr{D}}^*;\mathbb{R}^n)$,  $\Psi\in\cC^\infty(\overline{\mathscr{D}}^*)$ we have
\begin{align*}
\int_{\mathscr D} \left[ m \cdot \partial_t \Phi + \left(\sigma +\frac{2}{n}e\id+q\id\right):\nabla\Phi \right] \ dx \ dt -\int_0^T\int_{\partial\mathscr{D}(t)}q\Phi\cdot\vec{n}\ dx\ dt & \\ +\int_{\Omega}\rho_0(x)v_0 (x)\cdot \Phi(x,0)\ dx&=0,\\
\int_{\mathscr D} \left[w-gte_n\right]\cdot\nabla \Psi\:dx\:dt&=0,\\
\int_{\mathscr D} \left[\mu \partial_t \Psi + m\cdot\nabla\Psi \right] \ dx \ dt +\int_{\Omega} \rho_0 (x) \Psi(x,0)\ dx&=0,\\
\int_{\mathscr D} \left[e \partial_t \Psi + h\cdot\nabla\Psi\right] \ dx \ dt -\int_0^T gt \int_{\partial\mathscr{D}(t)}q {\Psi}e_n\cdot \vec{n}\ dx\ dt&\\+\int_{\Omega} e_0(x) \Psi(x,0)\ dx&=-\nu[\Psi].
\end{align*}
\end{definition}

Concerning the change of coordinates introduced in Section \ref{sec:change_of_coordinates} we would like to sketch the following equivalence.
\begin{lemma}\label{lem:translation} Let $(\rho_0,v_0)$ be as in Definition \ref{def:weaksols}. The tuple
$(\rho,v,p)$ is a solution with initial data $(\rho_0,v_0)$ in the sense of Definition \ref{def:weaksols} if and only if $(\mu,v,q)$ defined through \eqref{eq:transf} is a solution with initial data $(\rho_0,v_0)$ in the sense of Definition \ref{def:weaksolstrans}. Moreover, there holds \eqref{eq:translation_of_dissipation_measures} in duality with respect to $\cC^\infty_c(\overline{\Omega}\times[0,T))$.
\end{lemma}
\begin{proof}

Let $\Phi\in \cC^\infty_c(\overline{\Omega}\times[0,T);\mathbb{R}^n)$,  $\Psi_1,\Psi_2\in \cC^\infty_c(\overline{\Omega}\times[0,T))$,  $\Psi_2\geq 0$. We will show that the equivalence of weak solutions and dissipation measures follows straightforwardly by applying the transformations $y=x+\frac{1}{2}gt^2e_n$ and \eqref{eq:transf} directly in the integrals in the respective definitions. Indeed, we can begin by noting that the left-hand sides of the second equations from Definitions \ref{def:weaksols} and \ref{def:weaksolstrans} coincide trivially.

Furthermore, if one defines
$$\tilde\Phi(y,t):=\Phi(y-gt^2/2e_n,t)=\Phi(x,t),$$
one has
\begin{align*}
    \partial_t\tilde\Phi(y,t)=\frac{d}{dt}(\Phi(y-gt^2/2e_n,t))=\partial_t\Phi(x,t)-gt\partial_n \Phi(x,t),
\end{align*}
similarly for $\Psi_1$ and $\Psi_2$.

Then, the second easiest terms to check are the left-hand sides of the third equations from both definitions, since from \eqref{eq:transf} one has for almost every $(x,t)\in\Omega\times(0,T)$ that
\begin{multline*}
    \mu (y,t)\partial_t \tilde\Psi_1(y,t) + (\mu(y,t) w(y,t))\cdot\nabla\tilde\Psi_1(y,t)=\\
    \rho (x,t)\partial_t \Psi_1(x,t)-gt\rho (x,t)\partial_n\Psi_1(x,t) +gt\rho (x,t)\partial_n\Psi_1(x,t) + (\rho(x,t) v(x,t))\cdot\nabla\Psi_1(x,t)\\=
    \rho (x,t)\partial_t \Psi_1(x,t) + (\rho(x,t) v(x,t))\cdot\nabla\Psi_1(x,t).
\end{multline*}

Next, for the equivalence of the left-hands sides of the first lines of the two definitions, we write
\begin{align*}
    (\mu&(y,t) w(y,t)) \cdot \partial_t \tilde\Phi(y,t) + (\mu(y,t) w(y,t)\otimes w(y,t)+q(y,t)\id) :\nabla\tilde\Phi(y,t)\\
    &=(\rho(x,t) (v(x,t)+gt e_n)) \cdot \partial_t \Phi(x,t)-gt(\rho(x,t) (v(x,t)+gt e_n)) \partial_n\Phi(x,t)\\
    &\phantom{==}+ (\rho(x,t)(v(x,t)+gt e_n)\otimes(v(x,t)+gt e_n)+p(x,t)\id ) :\nabla\Phi(x,t)\\
    &=(\rho(x,t) v(x,t)) \cdot \partial_t \Phi(x,t) + (\rho(x,t) v(x,t)\otimes v(x,t)+p(x,t)\id ) :\nabla\Phi(x,t)\\
    &\phantom{==}-g\rho(x,t)\Phi_n(x,t) + \left(\rho(x,t)\partial_t(gt\Phi_n(x,t))+\rho(x,t)v(x,t)\cdot\nabla(gt\Phi_n(x,t)) \right).
\end{align*}
After integrating over the whole domain, the integral of the terms in the last parenthesis is zero, by applying the third equation in Definition \ref{def:weaksols} with $\Psi_1:=gt\Phi_n$, hence the equivalence follows.

Finally, to see that the dissipation measures coincide, through a similar calculation as above, we note that
\begin{align*}
    &\frac{\mu|w|^2}{2} \partial_t \tilde\Psi_2 + \left(\left(\frac{\mu|w|^2}{2}+q\right)w\right)\cdot\nabla\tilde\Psi_2= \rho\left(\frac{|v|^2}{2} +gt v_n+\frac{g^2t^2}{2} \right)\partial_t \Psi_2 
    \\&\hspace{115pt}+ \left(\rho\left(\frac{|v|^2}{2} +gt v_n+\frac{g^2t^2}{2}\right)+p \right) v \cdot \nabla\Psi_2 + gt p \partial_n\Psi_2
    \\&\phantom{=}=\left(\frac{\rho|v|^2}{2} +\rho g x_n\right) \partial_t \Psi_2 + \left(\left(\frac{\rho|v|^2}{2}+\rho g x_n+p\right)v\right)\cdot\nabla\Psi_2
    \\&\hspace{40pt}+\rho\partial_t\left(\frac{g^2 t^2}{2}\Psi_2\right)+\rho v\cdot\nabla\left(\frac{g^2 t^2}{2}\Psi_2\right)
    -\rho\partial_t\left(gx_n\Psi_2\right)+\rho v\cdot\nabla\left(gx_n\Psi_2\right)
    \\&\hspace{40pt}+(\rho v) \cdot \partial_t (gt\Psi_2 e_n) + (\rho v\otimes v+p\id ) :\nabla(gt\Psi_2 e_n)-g\rho gt\Psi_2. 
\end{align*}
After integration, it is easy to check that the terms in the last two lines
yield precisely
$$-\int_{\Omega} \rho_0(x) g x_n\Psi_2(x,0)\ dx,$$
by applying
the first and third equations of Definition \ref{def:weaksols} with $\Phi:=gt\Psi_2 e_n$ and $\Psi_1:=\frac{g^2t^2}{2}\Psi_2-gx_n\Psi_2$.
Thus we have shown that
\begin{multline*}
    \int_0^T\int_{\Omega} \left[\left(\frac{\rho|v|^2}{2} +\rho g x_n\right) \partial_t \Psi_2 + \left(\left(\frac{\rho|v|^2}{2}+\rho g x_n+p\right)v\right)\cdot\nabla\Psi_2\right] \ dx \ dt \\+\int_{\Omega} \left(\frac{\rho_0(x)|v_0(x)|^2}{2} +\rho_0(x) g x_n\right)\Psi_2(x,0)\ dx =\int_{\Omega} \frac{\rho_0(x)|v_0(x)|^2}{2}\tilde \Psi_2(x,0)\ dx\\-\int_0^T gt\int_{\partial\mathscr{D}(t)}q \tilde{\Psi}_2e_n\cdot \vec{n}\ dx\ dt+\int_{\mathscr D} \left[\frac{\mu|w|^2}{2} \partial_t \tilde\Psi_2 + \left(\left(\frac{\mu|w|^2}{2}+q\right)w\right)\cdot\nabla\tilde\Psi_2\right] \ dx \ dt. 
\end{multline*}
By formal partial integration and \eqref{eq:boundary_condition_w} one sees that the latter is indeed the distributional formulation provided by Definitions \ref{def:weaksols} and \ref{def:weaksolstrans} of the equality stated in \eqref{eq:translation_of_dissipation_measures}.
\end{proof}

Finally, we state an exact relation between the total energy of the convex integration solution in original variables and a quantity depending linearly on their subsolution in accelerated variables. 
\begin{lemma}\label{lem:total_energy}
Let $z$ be a subsolution with pressure $q$, dissipation measure $\nu$ and initial data $(\rho_0,v_0,e_0=1/2\rho_0\abs{v_0}^2)$. Let $(\rho_{sol},v_{sol})$ denote the solutions induced by Theorem \ref{thm:main_theorem} and transformation \eqref{eq:transf}, and assume that $\nu$ is not only a measure, but an $L^1$ function. Then the total energy of the solutions satisfies
\begin{align*}
    \int_\Omega \frac{1}{2}\rho_{sol}\abs{v_{sol}}^2+\rho_{sol}gx_n\:dx=\int_{\Omega+\frac{1}{2}gt^2}e-m_ngt+\mu gy_n\:dy.
\end{align*}
\end{lemma}
\begin{proof}
Let us first note that, similarly to  \cite[Remark 1.2 b)]{GK-EE}, if one assumes that $\nu$ is in addition an $L^1$ function, one can show proceeding as in \cite[Lemma 8]{DeL-Sz-Adm} that for any $\phi\in L^2(\Omega)$, the map $[0,T)\ni t\mapsto \int_\Omega \left(\frac{1}{2}\rho_{sol}\abs{v_{sol}}^2+\rho_{sol}gx_n\right)\phi\:dx$ can be assumed to be continuous. Thus, for fixed $t\in[0,T)$ we may formally use a test function $\Psi_2(x,\tau):=\mathbbm{1}_{[0,t]}(\tau)\mathbbm{1}_\Omega(x)$ in Definition \ref{def:weaksols}, by a standard mollification argument and passing to the limit. One may apply a similar argument concerning the expression $e-m_ngt+\mu gy_n$.

We denote by $\nu_{sol}$ the dissipation measure of the induced solutions $(\rho_{sol},v_{sol})$, by $\tilde{\nu}_{sol}$ the dissipation measure of the transformed solutions $(\mu_{sol},w_{sol})$, and $\tilde{\Psi}_2(y,t)=\Psi_2(x-1/2gt^2e_n,t)$. Lemma \ref{lem:translation}, in particular \eqref{eq:translation_of_dissipation_measures}, and Theorem \ref{thm:main_theorem} then imply
\begin{align*}
    \nu_{sol}[\Psi_2]=\tilde{\nu}_{sol}[\tilde{\Psi}_2]=\nu[\tilde{\Psi}_2].
\end{align*} 
Thus, Definitions \ref{def:weaksols} and \ref{def:weaksolslin} yield
\begin{align}\label{eq:totenergytest}
\begin{split}
   - \int_\Omega &\frac{1}{2}\rho_{sol}\abs{v_{sol}}^2+\rho_{sol}gx_n\:dx +  \int_\Omega \rho_{0}gx_n\:dx \\
   &\hspace{50pt}=
   - \int_{\mathscr{D}(t)} e \: dy - \int_0^t g\tau  \int_{\partial\mathscr{D}(\tau)}q e_n\cdot \vec{n}\ dy\ d\tau.
   \end{split}
\end{align}
On the other hand, through a similar argument as above, one may also use $\Phi(y,\tau):=g\tau\Psi_2 e_n=g\tau \mathbbm{1}_{[0,t]}(\tau)e_n$ as a test function for the first equation of Definition \ref{def:weaksolslin} to get that
\begin{align*}
    \int_0^t g\tau \int_{\partial\mathscr{D}(\tau)}q e_n\cdot \vec{n}\ dy\ d\tau &= \int_{\mathscr{D}} m_n \partial_\tau(g\tau\Psi_2) \ dy \ d\tau \\
    &= -\int_{\mathscr{D}(t)} m_n gt \ dy + \int_{\mathscr{D}} m \cdot \nabla (gy_n\Psi_2) \ dy \ d\tau \\
    &= -\int_{\mathscr{D}(t)} m_n gt \ dy -\int_\Omega \rho_{0}gy_n\: dy-\int_{\mathscr{D}} \mu \partial_\tau (gy_n\Psi_2) \ dy \ d\tau
    \\ &= -\int_{\mathscr{D}(t)} m_n gt \ dy -\int_\Omega \rho_{0}gy_n\: dy + \int_{\mathscr{D}(t)} \mu g y_n \ dy,
\end{align*}
where in the penultimate step we also used $\Psi:=gy_n\Psi_2$ in the third equation of Definition  \ref{def:weaksolslin}. Combining the above with \eqref{eq:totenergytest} finishes the proof.
\end{proof}

\section*{Acknowledgements} 

J. J. Kolumb\'an 
was supported by the J\'anos Bolyai Research Scholarship of the Hungarian Academy of Sciences. 
~B.G. is funded by the Deutsche Forschungsgemeinschaft (DFG, German Research Foundation) under Germany's Excellence Strategy EXC 2044-390685587, Mathematics Münster: Dynamics-Geometry-Structure.
In parts the research has also been carried out at Universidad Aut\'onoma de Madrid and Max Planck Institute for Mathematics in the Sciences, where B.G. acknowledges the support through Mar\'ia Zambrano Grant CA6/RSUE/2022-00097 and the great hospitality of MPI Leipzig.
The authors would like to thank Konstantin Kalinin for interesting discussions related to \cite{Kalinin_Menon_Wu}.

\bibliographystyle{abbrv}
\bibliography{bib}

\vspace{10pt}

\noindent Bj\"orn Gebhard\\
Mathematics M\"unster, University of M\"unster, 48149 Münster, Germany\\
bjoern.gebhard@uni-muenster.de

\vspace{10pt}

\noindent J\'ozsef J. Kolumb\'an \\
Budapest University of Technology and Economics, Department of Analysis and Operations Research, 1111 Budapest, M\H uegyetem rkp. 3, Hungary\\
jkolumban@math.bme.hu

\end{document}